\documentclass[11pt,a4paper,reqno]{amsart}
\usepackage[english]{babel}
\usepackage[T1]{fontenc}
\usepackage{verbatim}
\usepackage{palatino}
\usepackage{amsmath}
\usepackage{mathabx}
\usepackage{amssymb}
\usepackage{amsthm}
\usepackage{amsfonts}
\usepackage{graphicx}
\usepackage{esint}
\usepackage{color}
\usepackage{mathtools}

\usepackage[colorlinks = true, citecolor = black]{hyperref}
\title[Extensions and corona decompositions of intrinsic Lipschitz graphs]{Extensions and corona decompositions of low-dimensional intrinsic Lipschitz graphs in Heisenberg groups}
\author{Daniela Di Donato}
\author{Katrin F\"assler}
\address{Department of Mathematics and Statistics\\
University of Jyv\"{a}skyl\"{a}, P.O. Box 35 (MaD)\\FI-40014 University of Jyv\"{a}skyl\"{a}\\
Finland}
 \email{daniela.d.didonato@jyu.fi}
\email{katrin.s.fassler@jyu.fi}
\thanks{D.D.D. is partially supported by the Academy of Finland (grant
288501 `\emph{Geometry of subRiemannian groups}' and by grant
322898 `\emph{Sub-Riemannian Geometry via Metric-geometry and
Lie-group Theory}') and by the European Research Council
 (ERC Starting Grant 713998 GeoMeG `\emph{Geometry of Metric Groups}').  K.F. is supported by the Academy of Finland through the grant 321696
`Singular integrals, harmonic functions, and boundary regularity
in Heisenberg groups'. }

\newcommand{\R}{\mathbb{R}}
\newcommand{\W}{\mathbb{W}}
\newcommand{\He}{\mathbb{H}}
\newcommand{\N}{\mathbb{N}}

\newcommand{\Z}{\mathbb{Z}}

\newcommand{\calT}{\mathcal{T}}

\newcommand{\calD}{\mathcal{D}}

\newcommand{\V}{\mathbb{V}}

\def\Barint_#1{\mathchoice
          {\mathop{\vrule width 6pt height 3 pt depth -2.5pt
                  \kern -8pt \intop}\nolimits_{#1}}%
          {\mathop{\vrule width 5pt height 3 pt depth -2.6pt
                  \kern -6pt \intop}\nolimits_{#1}}%
          {\mathop{\vrule width 5pt height 3 pt depth -2.6pt
                  \kern -6pt \intop}\nolimits_{#1}}%
          {\mathop{\vrule width 5pt height 3 pt depth -2.6pt
                  \kern -6pt \intop}\nolimits_{#1}}}

\numberwithin{equation}{section}

\theoremstyle{plain}
\newtheorem{thm}[equation]{Theorem}

\newtheorem{lemma}[equation]{Lemma}

\newtheorem{cor}[equation]{Corollary}
\newtheorem{proposition}[equation]{Proposition}

\theoremstyle{definition}

\newtheorem{definition}[equation]{Definition}

\theoremstyle{remark}
\newtheorem{remark}[equation]{Remark}

\newcommand{\nref}[1]{(\hyperref[#1]{#1})}

\begin{document}

\begin{abstract}
This note concerns low-dimensional intrinsic Lipschitz graphs, in
the sense of Franchi, Serapioni, and Serra Cassano, in the
Heisenberg group $\mathbb{H}^n$, $n\in \mathbb{N}$.  For $1\leq
k\leq n$, we show that every intrinsic $L$-Lipschitz graph over a
subset of a $k$-dimensional horizontal subgroup $\mathbb{V}$ of
$\mathbb{H}^n$ can be extended to an intrinsic $L'$-Lipschitz
graph over the entire subgroup $\mathbb{V}$, where $L'$ depends
only on $L$, $k$, and $n$. We further prove that $1$-dimensional
intrinsic $1$-Lipschitz graphs in $\mathbb{H}^n$, $n\in
\mathbb{N}$, admit corona decompositions by intrinsic Lipschitz
graphs with smaller Lipschitz constants. This complements results
that were known previously only in the first Heisenberg group
$\mathbb{H}^1$. The main difference to this case arises from the
fact that for $1\leq k<n$, the complementary vertical subgroups of
$k$-dimensional horizontal subgroups in $\mathbb{H}^n$ are not
commutative.
\end{abstract}

\maketitle

\section{Introduction}

This note deals with low-dimensional intrinsic Lipschitz graphs in
Heisenberg groups. The \emph{$n$-th Heisenberg group
$\mathbb{H}^n$} is the set $\mathbb{R}^{2n+1}$ with the group
product ``$\cdot$'' given by
\begin{displaymath}
(x_1,\ldots,x_{2n},t)\cdot
(x_1',\ldots,x_{2n}',t')=\left(x_1+x_1',\ldots,x_{2n}+x_{2n}',t+t'+\tfrac{1}{2}\sum_{i=1}^n
x_i x_{n+i}'-x_i'x_{n+i}\right)
\end{displaymath}
for
$(x_1,\ldots,x_{2n},t),(x_1',\ldots,x_{2n}',t')\in\mathbb{R}^{2n+1}$.
We equip $\mathbb{H}^n$ with the left-invariant metric
\begin{equation}\label{eq:metric}
d(p,q):=\|q^{-1}\cdot p\|,\quad p,q\in \mathbb{H}^n,
\end{equation}
where $\|(x,t)\|:=\max\{|x|,\sqrt{|t|}\}$ and $|\cdot|$ denotes
the usual Euclidean norm on $\mathbb{R}^{2n}$.

\emph{Intrinsic Lipschitz graphs (iLG)} in $\mathbb{H}^n$ were
introduced by Franchi, Serapioni, and Serra Cassano in \cite{FSS}.
The definition of \emph{codimension-$1$} iLG is motivated by their
appearance in connection with a structure theorem for sets of
finite perimeter \cite{FSSC}, see also
\cite{NY,naor2020foliated,orponen2020subelliptic} for various
applications of such sets. The definition of iLG makes  perfect
sense also for \emph{low dimensions}, but there are fewer works
that study specifically low-dimensional iLG. Recently, they have
appeared in \cite{antonelli2020intrinsically,antonellimerlo2}. To
the best of our knowledge, $1$-dimensional iLG in $\mathbb{H}^1$
were first applied by Orponen and the second author in
\cite{fssler2019singular} to prove the boundedness of certain
singular integral operators on regular curves in $\mathbb{H}^1$.
The results of the present paper constitute a first step towards
the generalization of \cite{fssler2019singular} to higher
dimensional Heisenberg groups. At the same time, we believe that
the results are of independent interest in geometric measure
theory, as they complement the list of fundamental properties that
low-dimensional iLG share with Euclidean Lipschitz graphs. This is
our first main result:

\begin{thm}[Intrinsic Lipschitz extension]\label{t:MainIntro}
Let $n\in \mathbb{N}$, $k\in \{1,\ldots,n\}$, and assume that
$\mathbb{V}$ is a $k$-dimensional horizontal subgroup of
$\mathbb{H}^n$ with complementary vertical subgroup $\mathbb{W}$.
Then, for every $L\geq 0$, there exists a constant
$L'=L'(L,k,n)\geq 0$ such that every intrinsic $L$-Lipschitz
function $\phi:E \to \mathbb{W}$, defined on a subset $E\subset
\mathbb{V}$, can be extended to an intrinsic $L'$-Lipschitz
function $\overline{\phi}:\mathbb{V}\to\mathbb{W}$. Moreover, if
$k=n=1$, then one can take $L'= C(n) \max\{L,L^2\}$ where $C(n) \geq 1$ is a constant depending only on $n.$
\end{thm}
We defer the definitions to Section \ref{s:defns} and more precise
statements to Section \ref{s:extension}, and start by discussing
the connection between Theorem \ref{t:MainIntro} and other
Lipschitz extension results.

\medskip

A pair $(X,Y)$ of metric spaces has the \emph{Lipschitz extension
property} if there exists a constant $C>0$ such that, for every
$E\subset X$, every  Lipschitz function $f:E \to Y$  can be
extended to a Lipschitz function $\overline{f}:X\to Y$ with
Lipschitz constant $\mathrm{Lip}(\overline{f})\leq C\,
\mathrm{Lip}(f)$. It is known that the pair of metric spaces
$(\mathbb{R}^k,\mathbb{H}^n)$ has the Lipschitz extension property
if and only if $k\leq n$, see
\cite{MR1421823,MR2545863,FasslerMSc,MR2659687,MR2729637}. Theorem
\ref{t:MainIntro} is related to this result
 since every  intrinsic Lipschitz
function $\phi:\V\to \W$ (as in Theorem \ref{t:MainIntro}) is in
one-to-one correspondence with a (metrically defined) Lipschitz
function
\begin{displaymath}\Phi:E\subset (\mathbb{R}^k,|\cdot|) \to
(\mathbb{H}^n,d)\end{displaymath} for which
 $\Phi(E)$ is an \emph{intrinsic
graph} in the sense of Definition \ref{d:intr_Lip}, see Remark
\ref{r:iLGvsLip}. So the point of Theorem \ref{t:MainIntro} is to
extend the Lipschitz function $\Phi:E \to \mathbb{H}^n$ to
$\overline{\Phi}:\mathbb{R}^k \to \mathbb{H}^n$ in such a way that
the intrinsic graph structure of the image  is preserved.

\medskip

While Theorem \ref{t:MainIntro} thus yields a conclusion that does
not follow from the general Lipschitz extension property of
$(\mathbb{R}^k,\mathbb{H}^n)$ for $k\leq n$, our assumption is
also stron{\-}ger in that the image of the initially given,
partially defined Lipschitz map $\Phi$ is an intrinsic graph. This
additional information is very helpful in the construction of
Lipschitz extensions, and it led us to a proof for Theorem
\ref{t:MainIntro} that is different from the extension methods
used in \cite{MR2659687,MR2729637}. The case $k=n=1$ was proven
before in \cite{fssler2019singular}, but it also follows, with a
different argument, from our proof of Theorem \ref{t:MainIntro}.
 A new phenomenon appears in higher dimensional
Heisenberg groups, where there is a qualitative difference between
the condition for $k$-dimensional iLG in the middle dimension
$k=n$ and in smaller dimensions $k< n$. We establish the case
$k=n$ of Theorem \ref{t:MainIntro} by applying a $C^{1,1}$ version
of Whitney's extension theorem by Glaeser \cite{MR101294} to the
last component of $\phi$. The bridge between \cite{MR101294} and
intrinsic Lipschitz graphs is provided by the infinitesimal
condition appearing in Proposition \ref{l:EquivTameDiffOpenInt}.
The extension theorem in the case $k<n$ can be deduced by suitably
embedding $k$-dimensional graphs into $n$-dimensional graphs and
applying the $k=n$ version of Theorem \ref{t:MainIntro}.

\medskip

Theorem \ref{t:MainIntro} complements extension results for
low-\emph{co}dimensional iLG in $\mathbb{H}^n$, proved in
\cite{MR2836591,NY} for  codimension $1$, and in
\cite{vittone2020lipschitz} for codimension $k\leq n$. The proofs
in \cite{MR2836591,NY} use an argument similar to the classical
McShane Lipschitz extension theorem, which is possible since
$1$-dimensional horizontal subgroups in $\mathbb{H}^n$ can be
equipped with an order structure. The extension result in
\cite{vittone2020lipschitz} is based on a new level set
description of low-codimensional iLG. Neither of these approaches
is available for low-dimensional iLG.

\medskip

Theorem \ref{t:MainIntro}, the results mentioned in the last
paragraph, and the Ra{\-}de{\-}ma{\-}cher-type theorems in
\cite{MR2836591,vittone2020lipschitz,antonelli2020intrinsically}
show that all intrinsic Lipschitz functions between complementary
homogeneous subgroups of $\mathbb{H}^n$ share two fundamental
properties with Euclidean Lipschitz functions: the extension
property and the almost everywhere differentiability. These are
crucial features for applications in geometric measure theory. The
second main result of the present paper establishes an additional
property for $1$-dimensional intrinsic $1$-Lipschitz graphs,
namely a corona decomposition by intrinsic Lipschitz graphs
(possibly over different subgroups) with smaller constants. The
corresponding result for Euclidean Lipschitz graphs plays a
crucial role in the theory of quantitative rectifiability and
singular integrals \cite{10.2307/24896220,MR1251061,DS1}.

\begin{thm}[Intrinsic Lipschitz corona
decomposition]\label{t:coronaIntro} For every $n\in \mathbb{N}$
and $\eta \in (0,1)$, every $1$-dimensional intrinsic
$1$-Lipschitz graph in $\mathbb{H}^n$ admits a corona
decomposition by $1$-dimensional intrinsic $\eta$-Lipschitz
graphs.
\end{thm}
A \emph{corona decomposition} of a $1$-dimensional (intrinsic)
Lipschitz graph $\Gamma$ is a hierarchical partitioning, called
\emph{coronization}, of $\R$ (or a $1$-dimensional horizontal
subgroup) into ``good'' and ``bad'' dyadic intervals,
 where the bad ones are controlled by a Carleson packing
 condition,
 and the good ones can be partitioned into a forest of trees satisfying suitable
 properties.
 In particular, each tree $\mathcal{T}$
 comes with an (intrinsic) Lipschitz graph $\Gamma_{\mathcal{T}}$
 with smaller Lipschitz constant that approximates $\Gamma$ well
at the resolution of the intervals in the tree.  A more precise
statement is given in  Theorem \ref{Theorem 3.15n>1}. Bearing in
mind potential applications to singular integral operators,
Theorem \ref{Theorem 3.15n>1} states the approximation in
parametric form, using maps defined on a common domain, rather
than intrinsic graphs over possibly different horizontal
subgroups.

 In $\mathbb{H}^1$, a version of Theorem \ref{Theorem 3.15n>1}
was known before, see \cite[Theorem 3.15]{fssler2019singular}.
Using related ideas in the context of  $\He^n$, we give a proof
for the case $n>1$.
 Theorem \ref{Theorem 3.15n>1} yields a corona decomposition for $1$-dimensional intrinsic $1$-Lipschitz
 maps.
By fixing the trees of dyadic intervals in the coronization, and
rescaling the components of the map, we also obtain a corona
decomposition for all intrinsic Lipschitz maps with constant
greater than
 $1$,
 as stated in Corollary \ref{Theorem 3.15+}. This generalizes
\cite[Corollary 3.22]{fssler2019singular}.

In order to prove Theorem \ref{Theorem 3.15n>1}, we start by
recalling the  corona decomposition for Euclidean Lipschitz graphs
given by David and Semmes in \cite{MR1251061}, which we then state
for convenience in a slightly different form, Theorem
\ref{t:DSAdaptation}. In a certain sense, it allows to approximate
a $1$-Lipschitz map by $\delta$-Lipschitz maps for given $\delta
\in (0,1)$, up to subtracting linear maps. A common challenge in
the proof of the extension result (Theorem \ref{t:MainIntro})  and
the corona decomposition (Theorem \ref{t:coronaIntro}) is the
presence of nonlinear terms in the intrinsic Lipschitz condition
in dimensions $1\leq k<n$, see Lemma \ref{l:IntrLipExplicit},
which are absent for $k=n$, and in particular  for $n=1$.

\medskip

 \textbf{Structure of the paper.} Most concepts relevant for
the paper are  introduced in Section \ref{s:defns}. Section
\ref{ss:TameVSIntrLip} contains standard computations related to
low-dimensional iLG in Heisenberg groups. Section
\ref{ss:Infinitesimal} provides an infinitesimal characterization
of $1$- and  $n$-dimensional entire iLG in $\mathbb{H}^n$. The
extension result, Theorem \ref{t:MainIntro}, is proven in Section
\ref{s:extension}. The corona decomposition, Theorem
\ref{t:coronaIntro}, is finally given in Section \ref{s:corona}.

\medskip

\textbf{Acknowledgments.} We would like to thank Tuomas Orponen
and an anonymous referee for helpful suggestions.

\section{Definitions}\label{s:defns}

\subsection{Homogeneous subgroups, projections, and intrinsic Lipschitz graphs}
Let $n\in \mathbb{N}$. To introduce the relevant concepts, we
first fix a \emph{horizontal subgroup} $\mathbb{V}$ of
$\mathbb{H}^n$ of dimension $k\in \{1,\ldots,n\}$, which is given
by a set of the form
\begin{displaymath}
\mathbb{V}=V \times \{0\} \subset \mathbb{R}^{2n+1},
\end{displaymath}
where $V$ is a $k$-dimensional isotropic subspace of the standard
symplectic space $\mathbb{R}^{2n}$. This is equivalent to say that
$V$ is a $k$-dimensional subspace of $\mathbb{R}^{2n}$ so that
$(\mathbb{V},\cdot)$ is an abelian group isomorphic to
$(\mathbb{R}^k,+)$, see for instance \cite[Section 2]{MR2955184}.
Equipped with the metric $d$  defined in \eqref{eq:metric}, the
subgroup $\mathbb{V}$ is isometric to $(\mathbb{R}^k,|\cdot|)$.
The \emph{complementary vertical subgroup} $\mathbb{W}$ is given
by the Euclidean orthogonal complement of $\mathbb{V}$, that is,
$\mathbb{W}=V^{\bot}\times \mathbb{R}$.

Every point $p\in \mathbb{H}^n$ has a unique decomposition as
\begin{displaymath}p=\pi_{\mathbb{V}}(p)\cdot
\pi_{\mathbb{W}}(p)\quad\text{with}\quad\pi_{\mathbb{V}}(p)\in
\mathbb{V}\text{ and }\pi_{\mathbb{W}}(p)\in\mathbb{W}.
\end{displaymath}

\begin{definition}\label{d:intr_Lip} Assume that $\mathbb{V}$ and $\mathbb{W}$ are
homogeneous subgroups of $\mathbb{H}^n$ as above. A map $\phi:
E\subset \mathbb{V} \to \mathbb{W}$ is said to be \emph{intrinsic
$L$-Lipschitz} for a constant $L\geq 0$ if
\begin{displaymath}
\|\pi_{\mathbb{W}}(\Phi(v')^{-1}\cdot \Phi(v))\|\leq L
\|\pi_{\mathbb{V}}(\Phi(v')^{-1}\cdot \Phi(v))\|,\quad v,v'\in E,
\end{displaymath}
where $\Phi:E\subset \mathbb{V}\to \mathbb{H}^n$ is the
\emph{graph map} defined by $\Phi(v):= v\cdot \phi(v)$. The
\emph{intrinsic graph} of $\phi$ is the set
\begin{displaymath}
\Gamma:=\{v\cdot \phi(v):\; v\in E\}\subset \mathbb{H}^n,
\end{displaymath}
and we say that $\Gamma$ is an \emph{intrinsic $L$-Lipschitz graph
(over $E\subset \mathbb{V}$)}.
\end{definition}

It follows from \cite[Lemma 2.1]{MR2955184} and the choice of the
metric $d$ that for any pair of complementary homogeneous
subgroups $(\mathbb{V},\mathbb{W})$ and
$(\mathbb{V}',\mathbb{W}')$ as above there exists an isometric
isomorphism $f:(\mathbb{H}^n,d)\to (\mathbb{H}^n,d)$ with the
properties that $f(\mathbb{V})=\mathbb{V}'$,
$f(\mathbb{W})=\mathbb{W}'$ and such that $f$ maps every intrinsic
$L$-Lipschitz graph over a subset in $\mathbb{V}$ to an intrinsic
$L$-Lipschitz graph over a subset in $\mathbb{V}'$. For this
reason it is not restrictive to assume, as we will in the
following unless otherwise stated, that
\begin{equation}\label{eq:coordinates}
\mathbb{V}=\{(x_1,\ldots,x_k,0,\ldots,0)\colon
(x_1,\ldots,x_k)\in\mathbb{R}^k\}\end{equation}and\begin{equation}\label{eq:coordinates2}
\mathbb{W}=\{(0,\ldots,0,x_{k+1},\ldots,x_{2n},t)\colon
(x_{k+1},\ldots,x_{2n},t)\in\mathbb{R}^{2n+1-k}\}.
\end{equation}

\begin{remark}\label{r:iLGvsLip}
It follows from \cite[Proposition 3.7]{MR3511465} that if
$\phi:E\subset \mathbb{V}\to \mathbb{W}$ is intrinsic Lipschitz,
then the associated graph map is a  Lipschitz function $\Phi$ from
$(E,d)$ to $(\mathbb{H}^n,d)$, or from $(E,|\cdot|)$ to
$(\mathbb{H}^n,d)$, if we identify $E$ with a subset of
$\mathbb{R}^k$, using the map
\begin{displaymath}
(x_1,\ldots,x_k,0,\ldots,0)\mapsto (x_1,\ldots,x_k).
\end{displaymath}
Conversely, if $\Phi:(E,|\cdot|)\to (\mathbb{H}^n,d)$ is
$L$-Lipschitz with respect to the given metrics, and we assume in
addition that it is of the form $\Phi(v):= v \cdot \phi(v)\in
\mathbb{V}\cdot \mathbb{W}$ for a map  $\phi: E\subset \mathbb{V}
\to \mathbb{W}$, then $\phi$ is intrinsic Lipschitz since, for all
$v,v'\in E$,
\begin{align*}
\|\pi_{\mathbb{W}}(\Phi(v')^{-1}\cdot \Phi(v))\|&\leq
 d(\Phi(v),\Phi(v')) + \|\pi_{\mathbb{V}}(\Phi(v')^{-1}\cdot \Phi(v))\|\\
&\leq L |v-v'|+ \|\pi_{\mathbb{V}}(\Phi(v')^{-1}\cdot \Phi(v))\|\\
&= (L+1) \|\pi_{\mathbb{V}}(\Phi(v')^{-1}\cdot \Phi(v))\|.
\end{align*}
\end{remark}

Once complementary subgroups as in \eqref{eq:coordinates} and
\eqref{eq:coordinates2} have been fixed, it is convenient to
identify $\phi:E\subset\mathbb{V} \to \mathbb{W}$ with a function
$\phi:E\subset\mathbb{R}^k \to \mathbb{R}^{2n+1-k}$ in the obvious
way. This identification applied to intrinsic Lipschitz functions
leads to the notion of \emph{tame maps} which we discuss in the
next section, see especially Propositions
\ref{p:FromIntrLipToTameGENERAL} and \ref{p:FromTameToIntrLip}.

\subsection{Tame maps} In connection with one-dimensional
intrinsic Lipschitz graphs in $\mathbb{H}^1$, \emph{tame maps}
from subsets of $\mathbb{R}^k$ to $\mathbb{R}^{2n+1-k}$ for
$k=n=1$ were introduced in \cite{fssler2019singular}. We extend
the definition to arbitrary $1\leq k\leq n$ with a slight
adaptation of the notation. Here   $\langle \cdot, \cdot \rangle$
denotes the standard scalar product on $\R^k$.

\begin{definition}\label{d:tame}
Let $k,n\in \N$, $1\leq k\leq n$, $E\subset\mathbb{R}^k$, and
$L_i\geq 0$ for $i\in \{k+1,\ldots,2n+1\}$. We say that a map
$\phi =(\phi_{k+1},\ldots, \phi _{2n+1}) :E \to \R^{2n+1-k}$ is
\emph{$( L_{k+1},\ldots,L_{2n+1})$-tame} if
\begin{enumerate}
\item\label{i:d1}  $\phi_i$ is  Euclidean $L_i$-Lipschitz for
$i=k+1,\ldots,2n$; \item\label{i:d2}  $\psi:=(\phi_{n+1},
\ldots,\phi _{n+k} )$ satisfies the following conditions:
\begin{enumerate} \item if $k=n$, then
\begin{align*}
\left|\phi_{2n+1}(y) - \phi_{2n+1}(x)  -\left\langle \psi (y),
y-x\right\rangle
 \right| &+ \left|\phi_{2n+1}(y) - \phi_{2n+1}(x)  - \left\langle \psi
(x),  y-x\right\rangle  \right|\\& \leq L_{2n+1} |x-y|^2,\quad
x,y\in E,
\end{align*}
\item if $k<n$, then
\begin{align*}
\Big|\phi_{2n+1}(y) - \phi_{2n+1}(x)  -\langle \psi (y),
y-x\rangle &-\tfrac{ 1 }{2}
 \sum_{i=k+1}^n \phi_i(y)\phi_{n+i} (x) -\phi_i(x)\phi_{n+i} (y)\Big| \\
 +\Big| \phi_{2n+1}(y) - \phi_{2n+1}(x)  - \langle \psi
(x), y-x\rangle &-\tfrac{ 1}{ 2} \sum_{i=k+1}^n
\phi_i(y)\phi_{n+i} (x) -\phi_i(x)\phi_{n+i} (y)\Big| \\&\leq
L_{2n+1} |y-x|^2,\quad x,y\in E.
\end{align*}
\end{enumerate}
\end{enumerate}
\end{definition}

\begin{remark}\label{r:2DimTameLip}
Condition \eqref{i:d2} in Definition \ref{d:tame}  is implied
(with twice the constant $L_{2n+1}$) by a one-sided version of
itself:
 \begin{equation*}
\left|\phi_{2n+1}(y) - \phi_{2n+1}(x)  -\left\langle \psi (y),
y-x\right\rangle
 \right| \leq L_{2n+1} |x-y|^2,\quad
x,y\in E,\quad\text{if }k=n, \end{equation*}
and\begin{equation*}
\Big|\phi_{2n+1}(y) -
\phi_{2n+1}(x) -\langle \psi (y), y-x\rangle -\tfrac{ 1 }{2}
 \sum_{i=k+1}^n \phi_i(y)\phi_{n+i} (x) -\phi_i(x)\phi_{n+i} (y)\Big| \leq L_{2n+1}
|y-x|^2,
\end{equation*}
for $x,y\in E$, if $k<n$.
\end{remark}

\begin{remark}\label{r:FromTameToLip} Condition \eqref{i:d2} in Definition
\ref{d:tame} implies by triangle inequality that
\begin{equation}\label{eq:TowardsLip} \left|\left\langle \psi (y)-\psi(x), \frac
{y-x}{|y-x|}\right\rangle\right|\leq L_{2n+1}|x-y|,\quad\text{for
}x,y\in E,x\neq y.\end{equation} If $k=1$, then $\psi=\phi_{n+1}$
is a real-valued function and \eqref{eq:TowardsLip} shows that
$\phi_{n+1}$ is $L_{2n+1}$-Lipschitz. In other words, if $k=1$,
then the Lipschitz continuity of $\phi_{n+1}$ is automatically
implied by part \eqref{i:d2} of Definition \ref{d:tame}, and part
\eqref{i:d1} holds with ``$L_{n+1}$'' replaced by
``$\min\{L_{n+1},L_{2n+1}\}$''.
\end{remark}

\begin{remark} For all $1\leq k\leq n$, Definition \ref{d:tame} implies that $\phi_{2n+1}$ is locally
Lipschitz. This is immediate in the case $k=n$, and if $k<n$, it
follows easily once one has observed that
\begin{align}\label{eq:reformulation}
\sum_{i=k+1}^n \phi_i(y)\phi_{n+i} (x)& -\phi_i(x)\phi_{n+i} (y)
\\&= \sum_{i=k+1}^n \left(\phi_i(y)-\phi_i(x)\right)\phi_{n+i} (x)
-\phi_i(x)\left(\phi_{n+i} (y)-\phi_{n+i}(x)\right).\notag
\end{align}
\end{remark}

\section{Elementary properties of tame maps}

\subsection{Connection between tame maps and intrinsic Lipschitz
functions}\label{ss:TameVSIntrLip}

In this section we explore the connection between intrinsic
Lipschitz functions (as in Definition \ref{d:intr_Lip}) and tame
maps (as in Definition \ref{d:tame}). It is this connection that
initially motivated Definition~\ref{d:tame}. Throughout this
section, we assume that $1\leq k\leq n$, and $\mathbb{V}$ is a
$k$-dimensional horizontal subgroup of $\mathbb{H}^n$ with
complementary vertical subgroup $\mathbb{W}$ with coordinate
expressions as in \eqref{eq:coordinates} and
\eqref{eq:coordinates2}. Slightly abusing notation, we identify a
set  $E\subset\mathbb{V}$ with $E\subset \mathbb{R}^k$, and
$\phi:E\to \mathbb{W}$  with $\phi:E\to \mathbb{R}^{2n+1-k}$.

\begin{lemma}\label{l:IntrLipExplicit}
A function $(\phi_{k+1},\ldots , \phi _{2n+1}) :E\subset
\mathbb{V} \to \mathbb{W}$ is intrinsic $L$-Lipschitz if and only
if
\begin{displaymath}
\| (0,\ldots,0, \phi_{ k+1} (v')-\phi_{ k+1} (v),\dots,
\phi_{2n}(v')-\phi_{2n}(v), H(v,v'))\|\leq L |v'-v|,\quad v,v'\in
E,
\end{displaymath}
where
\begin{displaymath}
H(v,v'):=  \phi_{2n+1}(v')-\phi_{2n+1}(v) + \langle \psi (v),
v'-v\rangle,\quad\text{if }k=n,
\end{displaymath}
and, if $k<n$, then
\begin{displaymath}
H(v,v'):=\phi_{2n+1}(v')-\phi_{2n+1}(v) + \langle \psi (v),
v'-v\rangle +\frac 1 2  \sum_{i=k+1}^n \phi_i(v')\phi_{n+i} (v)  -
\phi_i(v)\phi_{n+i} (v').
\end{displaymath}
\end{lemma}

\begin{proof}
We recall from Definition \ref{d:intr_Lip} that $\phi$ is
intrinsic $L$-Lipschitz  if and only if
\begin{equation}\label{eq:LipCond}
\| \pi_\W(\Phi (v)^{-1} \cdot \Phi (v') )\| \leq L \| \pi_\V(\Phi
(v)^{-1} \cdot \Phi (v')) \|,\quad v,v' \in E.
\end{equation}
The graph map $\Phi$ of $\phi$ is given by
\begin{equation*}
\begin{aligned}
\Phi (v)
=
 \Big(v, \phi_{k+1} (v),\dots ,\phi _{2n} (v), \phi _{2n+1} (v) + \tfrac{ 1}{ 2} \sum_{i=1}^{k} v_i \phi_{n+i} (v)\Big),
\end{aligned}
\end{equation*}
for $v=(v_1,\ldots,v_k)\in E$. Recalling that  $\psi(v)=(
\phi_{n+1}(v),\dots, \phi _{n+k} (v)),$ we observe
\begin{displaymath}
 \Phi (v)^{-1} \cdot \Phi (v')  =\left(v'-v, \phi_{k+1}(v')-\phi_{k+1}(v),\dots,
 \phi_{2n}(v')-\phi_{2n}(v),h(v,v')\right),
\end{displaymath}
where
\begin{displaymath}h(v,v'):=\phi_{2n+1}(v')-\phi_{2n+1}(v) +\tfrac{
1}{ 2} \langle v'-v,\psi(v) + \psi(v')\rangle,\quad \text{if }k=n,
\end{displaymath}
and, if $k<n$, then $h(v,v')$ is equal to
\begin{displaymath}
\phi_{2n+1}(v')-\phi_{2n+1}(v) +\tfrac{ 1}{ 2}\langle v'-v,\psi(v)
+ \psi(v')\rangle-\tfrac{ 1}{ 2} \sum_{i=k+1}^n
\left(\phi_i(v)\phi_{n+i} (v') -\phi_i(v')\phi_{n+i} (v)\right).
\end{displaymath}
Next, since $\pi_\V (x_1,\dots ,x_{2n},t) =  (x_1,\dots,
x_k,0,\dots ,0)$ and
\begin{equation*}
\pi_\W (x_1,\dots ,x_{2n},t) =  \Big(0,\dots, 0, x_{k+1},\dots
,x_{2n},t - \tfrac{ 1}{ 2} \sum_{i=1}^{k}x_i x_{n+i}\Big),
\end{equation*}
the right-hand side of \eqref{eq:LipCond} equals $L|v'-v|$, and
the left-hand side can be written as
\begin{displaymath}
\| \pi_\W(\Phi (v)^{-1} \cdot \Phi (v') )\|=  \| (0,\ldots,0,
\phi_{ k+1} (v')-\phi_{ k+1} (v),\dots,
\phi_{2n}(v')-\phi_{2n}(v), H(v,v'))\|,
\end{displaymath}
for $ v,v'\in E$, where $H(v,v')$ is defined as in the statement
of the lemma.
\end{proof}

Lemma \ref{l:IntrLipExplicit} provides a link between intrinsic
Lipschitz and tame maps. We formulate this in two separate
propositions.

\begin{proposition}\label{p:FromIntrLipToTameGENERAL}
If $\phi=(\phi_{k+1},\ldots , \phi _{2n+1})\colon E\subset
\mathbb{V} \to \mathbb{W}$ is intrinsic $L$-Lipschitz, then
$(\phi_{k+1},\ldots , \phi _{2n}, -\phi _{2n+1})$ is an
$(L_{k+1},\ldots,L_{2n+1})$-tame map from $E\subset \mathbb{R}^k$
to $\mathbb{R}^{2n+1-k}$ with
\begin{equation}\label{eq:LiForm}
L_i= \left\{\begin{array}{ll}L,&\text{for
}i=k+1,\ldots,2n,\\2L^2,&\text{for }i=2n+1.\end{array} \right.
\end{equation}
If $k=1$, then one can take $L_{n+1}=\min\{L,2L^2\}$.
\end{proposition}

\begin{proof} Once the tameness is established, the improvement for $k=1$ follows from \eqref{eq:LiForm} by Remark \ref{r:FromTameToLip}.
Hence it remains to prove the first part of the Proposition. Let
$\phi$ be an intrinsic $L$-Lipschitz function. According to Lemma
\ref{l:IntrLipExplicit} this means that
\begin{equation}\label{eq:spelled_out}
\| (0,\ldots,0, \phi_{ k+1} (v')-\phi_{ k+1} (v),\dots,
\phi_{2n}(v')-\phi_{2n}(v), H(v,v'))\|\leq L |v'-v|,\quad v,v'\in
E,
\end{equation}
where
\begin{displaymath}
H(v,v')=  \phi_{2n+1}(v')-\phi_{2n+1}(v) + \langle \psi (v),
v'-v\rangle,\quad\text{if }k=n,
\end{displaymath}
and
\begin{displaymath}
H(v,v')=\phi_{2n+1}(v')-\phi_{2n+1}(v) + \langle \psi (v),
v'-v\rangle +\tfrac{1}{ 2}  \sum_{i=k+1}^n \phi_i(v')\phi_{n+i}
(v) - \phi_i(v)\phi_{n+i} (v'),
\end{displaymath}
if $k<n$. Recalling that $\|(x,t)\|=\max\{|x|,\sqrt{|t|}\}$ for
$(x,t)\in \mathbb{R}^{2n}\times \mathbb{R}$, inequality
\eqref{eq:spelled_out} implies first that $\phi_i$ is a Euclidean
$L$-Lipschitz function for $i=k+1,\dots , 2n$, which is part
\eqref{i:d1} of the tameness condition in Definition \ref{d:tame}.
Second, we deduce from \eqref{eq:spelled_out} that
\begin{equation*}
\left| \phi_{2n+1}(v')-\phi_{2n+1}(v) +  \langle \psi (v),
v'-v\rangle   \right|^{1/2}  \leq L |v'-v|,\quad v,v'\in
E,\quad\text{if  }k=n,
\end{equation*}
and
\begin{equation*}
\left| \phi_{2n+1}(v')-\phi_{2n+1}(v) +  \langle \psi (v),
v'-v\rangle  +\tfrac{1}{ 2 } \sum_{i=k+1}^n \phi_i(v')\phi_{n+i}
(v)
 - \phi_i(v)\phi_{n+i} (v') \right|^{1/2}  \leq L |v'-v|,
\end{equation*}
for $v,v'\in E$ if $k<n$. Hence $(\phi_2,\ldots , \phi _{2n},
-\phi _{2n+1})$ is $(L,\ldots L,2L^2)$-tame in both cases.
\end{proof}

We now consider the converse implication.

\begin{proposition}\label{p:FromTameToIntrLip}
 If $(\phi_{k+1},\ldots , \phi _{2n}, -\phi
_{2n+1})\colon E\subset\mathbb{R}^k \to \mathbb{R}^{2n+1-k}$ is an
$(L_{k+1},\ldots,L_{2n+1})$-tame map, then
$\phi=(\phi_{k+1},\ldots , \phi _{2n+1}) :E\subset \mathbb{V} \to
\mathbb{W}$ is intrinsic $L$-Lipschitz with
\begin{displaymath}
L:= \max\left\{|(L_{k+1},\ldots,L_{2n})|,\sqrt{L_{2n+1}}\right\}.
\end{displaymath}
\end{proposition}
\begin{proof}
If $(\phi_{k+1},\ldots , \phi _{2n}, -\phi _{2n+1})$ is
$(L_{k+1},\ldots,L_{2n+1})$-tame, we find by the first condition
in Definition \ref{d:tame} that for $i=k+1,\ldots,2n$, the
function $\phi_i$ is $L_i$-Lipschitz on $E$. Moreover, recalling
that $$\psi (v)=(\phi_{n+1}(v),\dots, \phi_{n+k} (v)),$$ the
second condition in the tameness definition for
$(\phi_{k+1},\ldots , \phi _{2n}, -\phi _{2n+1})$ reads as
follows:
 if $k=n$,
\begin{align*}
 \left| \phi_{2n+1}(v') - \phi_{2n+1}(v) +  \langle \psi (v'),
v'-v\rangle  \right| & + \left| \phi_{2n+1}(v') - \phi_{2n+1}(v) +
\langle \psi (v), v'-v\rangle  \right|\\& \leq L_{2n+1} |v'-v|^2,
\end{align*}
and,  if $k<n$,
\begin{align*}
|\phi_{2n+1}(v') - \phi_{2n+1}(v) +  \langle \psi (v'),
v'-v\rangle
 &+\tfrac{ 1}{ 2} \sum_{i=k+1}^n \phi_i(v')\phi_{n+i} (v) -\phi_i(v)\phi_{n+i} (v')| \notag\\
 +|\phi_{2n+1}(v') - \phi_{2n+1}(v)
+ \langle \psi (v), v'-v\rangle &+\tfrac{ 1 }{2} \sum_{i=k+1}^n
\phi_i(v')\phi_{n+i} (v) -\phi_i(v)\phi_{n+i} (v') |\notag\\& \leq
L_{2n+1} |v'-v|^2,
\end{align*} for all $v,v'\in E$.

 Using Lemma
\ref{l:IntrLipExplicit}, we conclude that $\phi:=(\phi_{k+1},\dots
, \phi _{2n+1}) :E \to \W$ is an intrinsic $L$-Lipschitz function
since its graph map satisfies
\begin{displaymath}
\|\pi_{\W}(\Phi(v')^{-1}\cdot \Phi(v))\|\leq
\max\left\{|(L_{k+1},\ldots,L_{2n})|,\sqrt{{L}_{2n+1}}\right\}
|v-v'|=L |v-v'|,\quad v,v'\in E.
\end{displaymath}
\end{proof}

\subsection{Infinitesimal condition for tame maps on open
sets}\label{ss:Infinitesimal} Tame maps defined on open
quasiconvex sets can be characterized by an infinitesimal
condition. This characterization will be applied in the proofs of
the main results of this paper, the extension and the corona
decomposition for low-dimensional intrinsic Lipschitz graph. We
first discuss the case $k=1$ and $n>1$, which will be used in the
proof of Theorem \ref{Theorem 3.15n>1}.
\begin{proposition}\label{propositionk=1} Assume that $n>1$.
Let $I\subset \mathbb{R}$ be an open interval, and let
$\phi=(\phi_2,\ldots,\phi_{2n+1}):I \to \mathbb{R}^{2n}$.
\begin{enumerate}
\item If $\phi$ is $(L_2,\ldots,L_{2n+1})$-tame, then $\phi_{i}$
is $L_{i}$-Lipschitz for $i=2,\ldots,2n$, and $\phi_{2n+1}$ is
differentiable almost everywhere on $I$, $\dot{\phi}_{2n+1}\in
\mathcal{L}^{\infty}_{loc}(I)$, and
\begin{equation}\label{derivata2n+1ora}
\dot \phi_{2n+1} = \phi_{n+1} +\tfrac{ 1}{2} \sum_{i=2}^n  \dot
\phi_i \phi_{n+i} -\phi_i \dot \phi_{n+i},\quad\text{a.e.\ on }I.
\end{equation}
\item Conversely, if $\phi_i$ is $L_i$-Lipschitz for
$i=2,\ldots,2n$,  $\phi_{2n+1}$ is locally Lipschitz, and
\eqref{derivata2n+1ora} holds, then $\phi$ is
$(L_2',\ldots,L_{2n+1}')$-tame with
\begin{displaymath}
L_i':=L_i\quad\text{for }i\neq 2n+1\quad\text{and}\quad
L_{2n+1}':= 2\left(L_{n+1}+ \sum_{i=2}^n L_{i}L_{n+i}\right).
\end{displaymath}
\end{enumerate}
\end{proposition}

\begin{proof} We assume first that $\phi$ is  $(L_2,\ldots,L_{2n+1})$-tame, in particular,
$\phi_i$ is a Lipschitz function on $I$ for $i=2,\ldots,2n$.
 Rademacher's theorem implies
that $\phi_i$ is differentiable almost everywhere on $I$ with
bounded derivative. Condition \eqref{i:d2} in Definition
\ref{d:tame} reads
\begin{equation}\label{comediventa}
\begin{aligned}
& \left| \frac{\phi_{2n+1}(y) - \phi_{2n+1}(x) }{y-x} - \phi_{n+1} (y) -\frac 1 2 \sum_{i=2}^n \frac{\phi_i(y)\phi_{n+i} (x) -\phi_i(x)\phi_{n+i} (y) }{y-x} \right| \\
& \,\, + \left| \frac{\phi_{2n+1}(y) - \phi_{2n+1}(x) }{y-x}
 - \phi_{n+1} (x) -\frac 1 2 \sum_{i=2}^n \frac{\phi_i(y)\phi_{n+i} (x) -\phi_i(x)\phi_{n+i} (y) }{y-x} \right| \leq L_{2n+1} |y-x|,
\end{aligned}
\end{equation}
for all $x,y\in I$ with $x\ne y$, and the formula
\eqref{eq:reformulation} for $k=1$ is
\begin{equation}\label{stimacomponenti}
 \phi_i(y)\phi_{n+i} (x) -\phi_i(x)\phi_{n+i} (y) = \phi_i(y)( \phi_{n+i} (x)-\phi_{n+i}(y)) -\phi_{n+i} (y) ( \phi_{i} (x)-\phi_{i}(y)).
\end{equation} Using these two facts, it is easy to see that  $
\dot \phi_{2n+1}$ exists almost everywhere on  $I$ and
\eqref{derivata2n+1ora} holds. In particular, $ \dot
\phi_{2n+1}\in \mathcal{L}^{\infty}_{loc}(I)$.

\medskip

 Conversely, assume that
 $\phi_i$ is an $L_i-$Lipschitz function for $i=2,\dots , 2n$ and  $ \phi_{2n+1}$ is a locally Lipschitz function
  satisfying \eqref{derivata2n+1ora}.
 Then, the corresponding one-sided version of \eqref{comediventa} is satisfied for
 ``$L_{n+1}+ \sum_{i=2}^n L_{i}L_{n+i}$'' instead of ``$L_{2n+1}$''.
 Indeed, for $x, y \in I $ with $x<y$, the expression
 \eqref{stimacomponenti} can be rewritten as
 \begin{equation}\label{stimacomponenti_2}
 \phi_i(y)\phi_{n+i} (x) -\phi_i(x)\phi_{n+i} (y) = - \phi_i(y)\int_x^y \dot{\phi}_{n+i}(s)\,ds+\phi_{n+i}(y)\int_x^y\dot{\phi}_i(s)\,ds,
\end{equation}
 and we obtain that
\begin{align*}
&\left| \phi_{2n+1}(y) - \phi_{2n+1}(x) - \phi_{n+1} (y) (y-x) -\frac 1 2 \sum_{i=2}^n \phi_i(y)\phi_{n+i} (x) -\phi_i(x)\phi_{n+i} (y) \right| \\
&\overset{\eqref{stimacomponenti_2}}{=}\left|\int_x^y
\dot{\phi}_{2n+1}(s)\,ds-\int_x^y
\phi_{n+1}(y)\,ds+\tfrac{1}{2}\sum_{i=2}^n \phi_i(y)\int_x^y
\dot{\phi}_{n+i}(s)\,ds-\phi_{n+i}(y)\int_x^y
\dot{\phi}_i(s)\,ds\right|\\
&\overset{\eqref{derivata2n+1ora} }{=} \left|\int_x^y \phi_{n+1}(s)-
\phi_{n+1}(y)+\tfrac{1}{2}\sum_{i=2}^n \dot{\phi}_i(s)[
\phi_{n+i}(s)-\phi_{n+i}(y)]+\dot{\phi}_{n+i}(s)[\phi_i(y)-\phi_i(s)]\,ds\right|\\
 & \leq \left(L_{n+1}+\sum _{i=2}^n L_{i}L_{n+i} \right)|y-x|^2,
\end{align*}
where in the last inequality we used the fact that $\phi _i$ is
$L_{i}-$Lipschitz for every $i=2,\dots, 2n$.
\end{proof}

If $n>1$, there is a fundamental difference between tame maps
$\phi:E\subset \mathbb{R}^k \to \mathbb{R}^{2n+1-k}$ for $k=n$ and
for $k<n$. This difference is visible already in part \eqref{i:d2}
of Definition \ref{d:tame}, where the expression for $k<n$
contains an additional summand compared to the one for $k=n$. The
simple form of tame maps if $k=n$ can be used to characterize them
by means of a gradient equation for the last component, at least
if $E$ is open and quasiconvex. Recall that a set
$U\subset\mathbb{R}^n$ is \emph{$C$-quasiconvex}
 for a constant $C\geq 1$ (with respect to the Euclidean distance)
if for all $x,y\in U$, there is a curve $\gamma$ connecting $x$ to
$y$ inside $U$ of Euclidean length $\mathrm{length}(\gamma)\leq C
|x-y|$.%
%

\begin{proposition}\label{l:EquivTameDiffOpenInt} Let $n\in \mathbb{N}$ and assume that $U$
is an open subset of $\mathbb{R}^n$. For a function
$\phi=(\phi_{n+1},\ldots,\phi_{2n+1}):U \to \mathbb{R}^{n+1}$ the
following holds:
\begin{enumerate}
\item If $\phi$ is $(L_{n+1},\ldots,L_{2n+1})$-tame, then
$\phi_{i}$, $i=n+1,\ldots, 2n$, is Euclidean $L_{i}$-Lipschitz and
$\phi_{2n+1}$ is differentiable on $U$ with Lipschitz continuous
gradient
\begin{equation}\label{derivata2n+1}
\nabla \phi_{2n+1} = (\phi_{n+1},\ldots,\phi_{2n}) \quad\text{on
}U.
\end{equation}
In particular, $\phi_{2n+1}\in C^{1,1}(U)$. \item If $U$ is
additionally assumed to be $C$-quasiconvex, if
$(\phi_{n+1},\ldots,\phi_{2n})$ is $L$-Lipschitz with respect to
the Euclidean metric, and $\phi_{2n+1}$ satisfies
\eqref{derivata2n+1}, then $\phi$ is
$(L_{n+1}',\ldots,L_{2n+1}')$-tame with
\begin{displaymath}
L_i':=L\quad\text{for }i\in \{n+1,\ldots,2n\}\quad\text{and}\quad
L_{2n+1}':= 2 C^2\, L.
\end{displaymath}
\end{enumerate}
\end{proposition}

\begin{remark}\label{r:InfinitesimalK<N}
If $k<n$, then one can still carry out the argument in the first
part of the proof of Proposition \ref{l:EquivTameDiffOpenInt} for
tame $\phi:U\subset \mathbb{R}^k \to \mathbb{R}^{2n+1-k}$, but the
conclusion is that $\phi_{2n+1}$ satisfies
\begin{displaymath}
\nabla \phi_{2n+1} =
\begin{pmatrix}\phi_{n+1}+\frac{1}{2}\sum_{i=k+1}^n \phi_{n+i} \,\partial_{x_1} \phi_i-\phi_i \,\partial_{x_1} \phi_{n+i}\\\vdots\\
\phi_{n+k}+\frac{1}{2}\sum_{i=k+1}^n \phi_{n+i} \, \partial_{x_k}
\phi_i -\phi_i\, \partial_{x_k}
\phi_{n+i}\end{pmatrix}\quad\text{almost everywhere on }U,
\end{displaymath}
cf.\ Proposition \ref{propositionk=1} for $k=1$. Since
$\phi_{i},\phi_{n+i}$, $i\in \{k+1,\ldots,n\}$, are merely
Lipschitz functions, they are only almost everywhere
differentiable and the derivatives are just bounded measurable
functions, so one cannot conclude that $\phi_{2n+1}$ is
$C^{1,1}(U)$.
\end{remark}

\begin{remark}\label{r:self-improvement}
Proposition \ref{l:EquivTameDiffOpenInt} yields a self-improvement
phenomenon for the tameness constant $L_{2n+1}$ of a
$(L_{n+1},\ldots,L_{2n+1})$-tame map
$(\phi_{n+1},\ldots,\phi_{2n+1}):U \to \mathbb{R}^{n+1}$ defined
on an open and $C$-quasiconvex set $U\subset \mathbb{R}^n$. By
assumption, such $(\phi_{n+1},\ldots,\phi_{2n})$ is Euclidean
$|(L_{n+1},\ldots,L_{2n})|$-Lipschitz on $U$, and by Proposition
\ref{l:EquivTameDiffOpenInt} (1), the last component $\phi_{2n+1}$
is differentiable on $U$ with $\nabla
\phi_{2n+1}=(\phi_{n+1},\ldots,\phi_{2n})$. It then follows from
part (2) of the same proposition that
$(\phi_{n+1},\ldots,\phi_{2n+1})$ is in fact tame with constants
\begin{displaymath}
L_i =
\left\{\begin{array}{ll}|(L_{n+1},\ldots,L_{2n})|,&i=n+1,\ldots,2n,\\2C^2
|(L_{n+1},\ldots,L_{2n})|,&i=2n+1.\end{array}\right.
\end{displaymath}
Hence the initially given tameness constant ``$L_{2n+1}$'' can be
replaced by $$\min\{L_{2n+1},2C^2 |(L_{n+1},\ldots,L_{2n})|\}.$$
In particular, if $U=\mathbb{R}^n$, then this holds with $C=1$.
\end{remark}

\begin{remark}
 The correspondence between intrinsic Lipschitz and tame maps relates Proposition \ref{l:EquivTameDiffOpenInt} to
earlier results by Magnani \cite{MR2659687} and the second author
\cite{FasslerMSc}, keeping in mind the connection to metric
Lipschitz functions explained in Remark \ref{r:iLGvsLip}. More
precisely, \cite[Theorem 1.1]{MR2659687} and \cite[Theorem
4.5]{MR2659687} provide a characterization of (locally) Lipschitz
functions $\Phi$ from (geodetically convex) subsets of Riemannian
manifolds into graded groups through a system of first order PDEs
known as \emph{weak contact equations}. This characterization
applies in particular in our setting, where the source space is
Euclidean space $\mathbb{R}^k$ and the target space is the
Heisenberg group $\mathbb{H}^n$. The purpose of Proposition
\ref{l:EquivTameDiffOpenInt} is to show that if $\Phi:\mathbb{R}^k
\to \mathbb{H}^n$ arises as graph map of an intrinsic Lipschitz
function $\phi$, and if $k=n$, then this characterization takes a
particularly simple form and leads to a gradient equation for the
last component of $\phi$ that holds in the classical sense
pointwise everywhere.  This generalizes an observation made  in
\cite{fssler2019singular}: the condition for a curve $\gamma$ in
$\mathbb{H}^1$ to be horizontal (or Lipschitz with respect to $d$)
simplifies if $s\mapsto \gamma(s)= (s,0,0)\cdot \phi(s,0,0)$ has
intrinsic graph form. Indeed, whereas the last component of a
Lipschitz curve $\gamma$ in $\mathbb{H}^1$ need not even be
everywhere differentiable, the last component of $\phi$ is
$C^{1,1}$ if $\gamma$ is Lipschitz.
\end{remark}

\begin{proof}[Proof of Proposition \ref{l:EquivTameDiffOpenInt}]
 We assume first that $\phi$ is  $(L_{n+1},\ldots,L_{2n+1})$-tame, in particular,
$\phi_i$ is a Lipschitz function on $U$ for $i=n+1,\ldots,2n$.
 Condition \eqref{i:d2} (a) in
Definition \ref{d:tame} and the fact that $U$ is open then imply
that $ \nabla\phi_{2n+1}$ exists on $U$ and \eqref{derivata2n+1}
holds.

For the converse implication, we assume in addition that $U$ is
$C$-quasiconvex. We claim that if
 $(\phi_{n+1},\ldots,\phi_{2n})$ is an $L$-Lipschitz function, and  \eqref{derivata2n+1}
 holds, then the tameness condition \eqref{i:d2} (a) in Definition \ref{d:tame} is satisfied
 with constant
 \begin{displaymath}
L_{2n+1}':= 2 C^2 L
 \end{displaymath}
 (instead of $L_{2n+1}$). According to Remark \ref{r:2DimTameLip}, it suffices
to verify the one-sided version of it (without the constant
``$2$''). To prove the latter, let $x$ and $y$ be  arbitrary
distinct points in $U$, and apply the $C$-quasiconvexity of $U$ to
find a curve $\gamma:[0,1]\to U$  with $\gamma(0)=x$,
$\gamma(1)=y$ and $\mathrm{length}(\gamma)\leq C |x-y|$. Since
$\gamma$ is a curve of finite length, we may without loss of
generality assume that the parametrization is Lipschitz
continuous. The fundamental theorem of calculus then yields for
$\psi:=(\phi_{n+1},\ldots,\phi_{2n})$ that
\begin{align*}
\phi_{2n+1}(y)-\phi_{2n+1}(x)-\langle \psi(y),y-x\rangle&=
\phi_{2n+1}(\gamma(1))-\phi_{2n+1}(\gamma(0))-\langle
\psi(\gamma(1)),\gamma(1)-\gamma(0)\rangle\\
&= \int_0^{1} (\phi_{2n+1}\circ \gamma)'(s)-\langle
\psi(\gamma(1)),\dot{\gamma}(s)\rangle\,ds\\
&= \int_0^1 \langle \nabla
\phi_{2n+1}(\gamma(s)),\dot{\gamma}(s)\rangle-\langle
\psi(\gamma(1)),\dot{\gamma}(s)\rangle \,ds\\
&\overset{\eqref{derivata2n+1}}{=} \int_0^1 \langle
\psi(\gamma(s))-\psi(\gamma(1)),\dot{\gamma}(s)\rangle\,ds.
\end{align*}
Taking absolute values on both sides, we conclude that
\begin{align*}
|\phi_{2n+1}(y)-\phi_{2n+1}(x)-\langle \psi(y),y-x\rangle|&\leq
\int_0^1
\mathrm{Lip}(\psi)|\gamma(s)-\gamma(1)|\,|\dot{\gamma}(s)|\,ds\\
&\leq \mathrm{Lip}(\psi) \,\mathrm{length}(\gamma)^2 \leq C^2 L
|x-y|^2.
\end{align*}
\end{proof}

\section{Extension results}\label{s:extension}

The core of this section are extension results for tame maps
$\phi:E\subset\mathbb{R}^k \to \mathbb{R}^{2n+1-k}$, first for
$k=n$ and then, as a corollary, for $k<n$. In the final
subsection, we use Propositions \ref{p:FromIntrLipToTameGENERAL}
and \ref{p:FromTameToIntrLip} to translate these results into an
extension theorem for intrinsic Lipschitz functions (Theorem
\ref{t:MainIntro} in the introduction).

\subsection{Extension of tame maps in the case
$k=n$}\label{ss:extkn}

A classical method for extending Lipschitz functions $f:E \to
\mathbb{R}^m$ from a closed set $E\subset\mathbb{R}^n$ to the
entire space $\mathbb{R}^n$ is based on a Lipschitz partition of
unity associated to a Whitney decomposition of the complement
$\mathbb{R}^n \setminus E$, see for instance
\cite[2.10]{MR2177410}. Variants of this approach are also at the
core of the Lipschitz extension theorems in
\cite{MR2200122,MR2729637}. To establish the main result of this
section, we apply a version of Whitney's extension theorem, so
that the proof is again, albeit indirectly, based on a Whitney
decomposition of $\mathbb{R}^n \setminus E$. The key observation
is that in our setting it suffices to apply Whitney's construction
\emph{to the last component} of the tame map. More precisely, we
will use a $C^{1,1}$ version of Whitney's extension theorem due to
Glaeser \cite{MR101294}, see for instance \cite[Definition 2.2,
(2.48), and Theorem 2.19]{MR2882877}, \cite[Lemma
10.70]{MR2868143}, and the references cited in \cite{MR3777638}.
Here $C^{1,1}(\mathbb{R}^n)$ is the space of $C^1(\mathbb{R}^n)$
functions with Lipschitz continuous gradients.

\begin{thm}[Glaeser's $C^{1,1}$ Whitney extension
theorem]\label{t:glaeser} Let $n\in \mathbb{N}$ and assume that
$E$ is a subset of $\mathbb{R}^n$. The following conditions for
functions $f:E\to \mathbb{R}$ and $\psi:E\to \mathbb{R}^n$ are
equivalent:
\begin{enumerate}
\item[(a)] there exists $\overline{f}\in C^{1,1}(\mathbb{R}^n)$
with $\overline{f}|_E = f$ and $(\nabla \overline{f})|_E=\psi$,
\item[(b)] for a constant $\lambda>0$ and all $x,y\in E$, the
following holds:
\begin{enumerate}
\item[(1)]  $|\psi(x)-\psi(y)|\leq \lambda |x-y|$, \item[(2)]
$|f(x)-f(y)-\langle \psi(x),x-y\rangle|\leq \lambda |x-y|^2$.
\end{enumerate}
\end{enumerate}
Moreover, if (b) holds, then $\overline{f}$ can be constructed so
that the Lipschitz constant of $\nabla \overline{f}$ satisfies
\begin{displaymath}
\inf \lambda\leq \mathrm{Lip}(\nabla \overline{f})\leq C(n) \inf
\lambda,
\end{displaymath}
where the $\inf$ ranges over all $\lambda$ satisfying (b), and
$C(n)$ is a constant depending only on $n$.
\end{thm}

\begin{remark}\label{r:ExtClos} Theorem \ref{t:glaeser} is stated
for arbitrary subsets $E$ of $\mathbb{R}^n$, and the same holds
true for our application in Theorem \ref{propK=N} and the
corollaries thereof. As observed in \cite[\S 1]{MR3777638}, if $f$
and $\psi$ satisfy condition (b) in Theorem \ref{t:glaeser} for a
set $E\subset \mathbb{R}^n$, then one can always extend them to
the closure $\overline{E}$ of $E$ so that inequalities (1) and (2)
in (b) are satisfied on $\overline{E}$ with the same constant
$\lambda$. (Extending $f$ and $\psi$ as continuous maps to the
closure is straightforward since $f$ is locally Lipschitz and
$\psi$ is Lipschitz; then it just remains to verify that the
inequalities (1) and (2) continue to hold on $\overline{E}$.)
Conversely, if $f$ and $\psi$ defined on $E$ satisfy condition
(a), then they can obviously be extended to $\overline{E}$ so that
(a) continues to hold. Thus the proof of Theorem \ref{t:glaeser}
is reduced to the case of closed sets.
\end{remark}

\begin{thm}\label{propK=N}
Let $n\in \mathbb{N}$. An $(L_{n+1},\ldots,L_{2n+1})$-tame map
$\phi:E \subset \mathbb{R}^n \to \mathbb{R}^{n+1}$ can be extended
to an $(L_{n+1}',\ldots,L_{2n+1}')$-tame map
$\overline{\phi}:\mathbb{R}^n \to \mathbb{R}^{n+1}$ such that
$\overline{\phi}|_E = \phi$ and
\begin{displaymath}
L_i':= C(n)\max\left\{|(L_{n+1},\ldots,L_{2n})|,L_{2n+1}
\right\},\quad\text{for }i\in\{n+1,\ldots,2n\},
\end{displaymath}
\begin{displaymath}
 L_{2n+1}':= 2 C(n)\max\{|(L_{n+1},\ldots,L_{2n})|,L_{2n+1}\}.
\end{displaymath}
\end{thm}

\begin{proof}[Proof of Theorem \ref{propK=N}]
Let $\phi$ be a tame map on $E\subset \mathbb{R}^n$ as in the
statement of the theorem. In order to extend $\phi$ to a tame map
defined on all of $\mathbb{R}^n$, we apply the $C^{1,1}$ Whitney
extension theorem
 the function $\phi_{2n+1}:E\to
\mathbb{R}$. The tameness conditions  in Definition \ref{d:tame}
ensure that the assumptions of Theorem \ref{t:glaeser} are
satisfied with
\begin{displaymath}
\lambda:= \max\left\{|(L_{n+1},\ldots,L_{2n})|,L_{2n+1}\right\}.
\end{displaymath}
Thus we find a $C^{1,1}$ function $\overline{\phi}_{2n+1}:\R^n \to
\R$ with ${\overline{\phi}_{2n+1}}|_E=\phi_{2n+1}$ whose gradient
is $C(n)\lambda$-Lipschitz and agrees with
$\psi=(\phi_{n+1},\ldots,\phi_{2n})$ on $E$. Then we simply define
\begin{equation}\label{derivata2n+1k2n2}
(\overline{\phi}_{n+1},\ldots,\overline{\phi}_{2n}):=\nabla
\overline{\phi}_{2n+1},
\end{equation} and observe that this extends $(\phi_{n+1},\ldots,\phi_{2n})$ from $E$ to
$\R^n$. Thus,
\begin{displaymath}
\overline{\phi}:=(\nabla
\overline{\phi}_{2n+1},\overline{\phi}_{2n+1})=(\overline{\phi}_{n+1},\ldots,\overline{\phi}_{2n},\overline{\phi}_{2n+1})
\end{displaymath}
is an extension of $\phi$ to the entire space $\mathbb{R}^n$.

Finally, we apply the infinitesimal characterization from
Proposition \ref{l:EquivTameDiffOpenInt} to conclude that
$\overline{\phi}$ is a tame map. Since the domain $U=\mathbb{R}^n$
is quasiconvex with constant $1$,  Proposition
\ref{l:EquivTameDiffOpenInt} (2) implies that $\overline{\phi}$ is
$(L_{n+1}',\ldots,L_{2n}',L_{2n+1}')$-tame with
\begin{displaymath}
L_i'= C(n)\lambda\quad \text{for all }i\in
\{n+1,\ldots,2n\}\quad\text{and}\quad L_{2n+1}'=2 C(n) \lambda.
\end{displaymath}
\end{proof}

\begin{remark}
Applying Theorem \ref{t:glaeser} to the last component
$\phi_{2n+1}$ of a tame map $\phi$ and extending
$\psi=(\phi_{n+1},\ldots,\phi_{2n})$ by formula
\eqref{derivata2n+1k2n2} ensures that
\begin{displaymath}
\nabla \overline{\phi}_{2n+1} =
(\overline{\phi}_{n+1},\ldots,\overline{\phi}_{2n+1})
\end{displaymath}
is satisfied \emph{by definition}. If, on the other hand, one
tried to extend $(\phi_{n+1},\ldots,\phi_{2n})$ first, then one
would have to make sure that the extension can arise as gradient,
and this would entail a further differential constraint for
$\overline{\phi}_i$, $i\in \{n+1,\ldots,2n\}$, cf.\ the related
\emph{isotropic mappings} appearing in
\cite{MR2659687,FasslerMSc}.
\end{remark}

\subsection{Extension of tame maps in the case
$k<n$}\label{ssextk<n}

In this section, we prove the extension result for tame maps
$\phi:E\subset \mathbb{R}^k \to \mathbb{R}^{2n+1-k}$ in the case
$k<n$. The situation is qualitatively different from the
middle-dimensional case $k=n$ discussed in the previous section.
Indeed, recall from Remark \ref{r:InfinitesimalK<N} that the last
component of an entire tame map $\phi:\mathbb{R}^k \to
\mathbb{R}^{2n+1-k}$ satisfies almost everywhere the nonlinear
gradient equation
\begin{displaymath}
\nabla \phi_{2n+1} =
\begin{pmatrix}\phi_{n+1}+\frac{1}{2}\sum_{i=k+1}^n \phi_{n+i}\, \partial_{x_1} \phi_i -\phi_i \,\partial_{x_1} \phi_{n+i}\\\vdots\\
\phi_{n+k}+\frac{1}{2}\sum_{i=k+1}^n  \phi_{n+i}\,\partial_{x_k}
\phi_i -\phi_i\, \partial_{x_k}
\phi_{n+i}\end{pmatrix}\quad\text{almost everywhere},
\end{displaymath}
so we are no longer in a setting where Whitney's extension theorem
is directly applicable. However, it turns out that the extension
in case $k<n$ can be reduced to the case $k=n$. This is best
understood if one thinks of intrinsic Lipschitz graphs instead of
tame maps. The idea is essentially that a $k$-dimensional
intrinsic Lipschitz graph in $\mathbb{H}^n$ for $k<n$ can be
embedded in an $n$-dimensional intrinsic Lipschitz graph. The
latter can be extended using Theorem \ref{propK=N}, and then it
remains to show that one can select a suitable $k$-dimensional
subset of it in order to obtain an extension of the original
graph.

\begin{thm}\label{propKminoreN}
Let $k,n\in \mathbb{N}$ with $1\leq k<n$. An
$(L_{k+1},\ldots,L_{2n+1})$-tame map $$\phi:E \subset \mathbb{R}^k
\to \mathbb{R}^{2n+1-k}$$ can be extended to an
$(L_{k+1},\ldots,L_n,L_{n+1}',\ldots,L_{2n+1}')$-tame map
$\overline{\phi}:\mathbb{R}^k \to \mathbb{R}^{2n+1-k}$ with
$\overline{\phi}|_E = \phi$ and
\begin{displaymath}L_i'= c_n \Big(1+\sum_{j=k+1}^n L_j^2\Big)^{\frac{1}{2}} \max  \left\{|(L_{n+1},\ldots,L_{2n})|,L_{2n+1}+\sum_{i=k+1}^n
L_{n+i}\min\{1,L_i\}\right\}
\end{displaymath}
for $i=n+1,\ldots,2n$, and
\begin{displaymath}
L_{2n+1}'=  c_n  \Big(1+\sum_{j=k+1}^n L_j^2\Big)\max
\left\{|(L_{n+1},\ldots,L_{2n})|,L_{2n+1}+\sum_{i=k+1}^n
L_{n+i}\min\{1,L_i\}\right\}
\end{displaymath}
for a constant $c_n$ that depends only on $n$.
\end{thm}

\begin{proof} Since $k<n$, the Lipschitz map $(\phi_{k+1},\ldots,\phi_n)$
has at least one component, and we can consider the associated
$k$-dimensional Lipschitz graph
\begin{equation*}
\Gamma ^{(\phi _{k+1}, \dots , \phi _n)} (E):= \{ (x , \phi
_{k+1}(x), \dots , \phi _n (x)) \, :\, x\in E  \} \subset \R^n.
\end{equation*}
The remaining components of $\phi$ are used to define $ f: \Gamma
^{(\phi _{k+1}, \dots , \phi _n)} (E)\subset \mathbb{R}^n \to
\R^{n+1} $ by
\begin{align}\label{definf}
&f(\eta _1,\dots , \eta_n):=
(f_{n+1}(\eta),\ldots,f_{2n}(\eta),f_{2n+1}(\eta)):=\\&\left(\phi
_{n+1}(\eta _1,\dots, \eta _k), \dots , \phi _{2n} (\eta _1,\dots,
\eta _k), \phi _{2n+1} (\eta _1,\dots, \eta _k) +\tfrac{ 1}{ 2}
\sum_{i=k+1}^{n} \eta_i\phi _{n+i}(\eta _1,\dots, \eta
_k)\right).\notag
\end{align}
Firstly, we show that $f$ is a tame map and so we can apply
Theorem \ref{propK=N} to find an extension $\bar f=(\bar f_{n+1},
\dots , \bar f_{2n+1}):\R^n \to \R^{n+1}$ of $f$  with the
corresponding tameness assumption satisfied. Second, if $(\bar
\phi _{k+1}, \dots , \bar \phi _{n})$ denotes a suitable Euclidean
Lipschitz extension of $( \phi _{k+1}, \dots , \phi _{n})$, we
prove that the map $\bar \phi : \R^k \to \R^{2n+1-k}$ given by
\begin{align}\label{goodextension}
\bar \phi (x):= \Biggl(&\bar \phi _{k+1}(x), \dots , \bar \phi _{n}(x),\bar f_{n+1}  (x , \bar \phi _{k+1}(x), \dots ,\bar \phi _n (x)) ,
\dots ,\bar f_{2n}  (x , \bar \phi _{k+1}(x), \dots , \bar \phi _n (x)),\notag \\
& \, \bar f_{2n+1}  (x , \bar \phi _{k+1}(x), \dots , \bar \phi _n
(x))- \tfrac{ 1}{ 2} \sum_{i=k+1}^{n} \bar\phi _{i}(x)\bar f_{n+i}
(x ,\bar \phi _{k+1}(x), \dots ,\bar \phi _n (x)) \Biggl),
\end{align}
 is the tame  extension of $\phi$ to $\R^k$ we are looking for.
We begin by proving that the map $f$ defined  in \eqref{definf} is
$(L_{n+1}'',\ldots,L_{2n+1}'')$-tame with
\begin{equation}\label{eq:L''}
L_i'':= L_i,\quad i\in \{n+1,\ldots,2n\},\quad\text{and}\quad
L_{2n+1}'':= L_{2n+1} + \sum_{i=k+1}^n L_{n+i}\min\{1,L_i\} .
\end{equation}
To see this, fix two points
 \begin{equation}\label{eq:p_q}
p=(x,\phi _{k+1}(x), \dots , \phi _n (x)), q=(y,\phi _{k+1}(y),
\dots , \phi _n (y)) \in \Gamma ^{(\phi _{k+1}, \dots , \phi _n)}
(E).
\end{equation}
 The components $ f_{n+1}, \dots , f_{2n}$ are clearly
Lipschitz since, for $i=n+1,\ldots,2n$, we have by the
$L_i$-Lipschitz continuity of $\phi_i$ that
\begin{align*}
|f_i(p)-f_i(q)|&=\left|\phi_i\left(x,\phi _{k+1}(x), \dots , \phi
_n (x)\right)-\phi_i\left(y,\phi _{k+1}(y), \dots , \phi _n
(y)\right)\right|\leq L_i   |p-q|.
\end{align*}
It remains to check that the tameness condition \eqref{i:d2} in
Definition \ref{d:tame} holds with constant $L_{2n+1}'''$.
Using the notation $\psi=(\phi_{n+1},\dots,  \phi _{n+k})$ and
recalling the expressions for $p$ and $q$ given in \eqref{eq:p_q},
we have that
\begin{equation*}
\begin{aligned}
I(p,q):=\big| f_{2n+1}(q) - f_{2n+1}(p) &-\left\langle (f_{n+1}(q)
, \ldots, f_{2n} (q)), ( q_1-p_1, \ldots , q_{n} -p_{n}
)\right\rangle
 \big| \\
  = \Biggl| \phi_{2n+1}(y) - \phi_{2n+1}(x) &-\langle \psi (y), y-x \rangle
\\&+\tfrac{1}{2}\sum_{i=k+1}^n \phi_i(y)\phi_{n+i}(y)-
\phi_i(x)\phi_{n+i}(x)
 -2 \phi _{n+i}(y)(\phi _{i}(y)- \phi_{i}  (x))  \Biggl|.
\\
  \leq \Biggl| \phi_{2n+1}(y)- \phi_{2n+1}(x) &-\langle \psi (y), y-x \rangle -\tfrac{ 1}{ 2}  \sum_{i=k+1}^{n} \phi _{i}(y)\phi _{n+i}(x)- \phi _{i}(x)\phi _{n+i}(y) \Biggl| \\
& + \tfrac{1}{2}\Biggl|  \sum_{i=k+1}^{n}    \phi _{n+i}(y) ( \phi _{i}(x)- \phi _{i}(y)) - \phi _{n+i}(x) ( \phi _{i}(x)- \phi _{i}(y))  \Biggl|, \\
\end{aligned}
\end{equation*}
and analogously with the roles of $p$ and $q$ reverted. Summing up
the two expressions, we obtain by the
$(L_{k+1},\ldots,L_{2n+1})$-tameness of $\phi$ that
\begin{align*}
I(p,q)+I(q,p) & \leq L _{2n+1}|y-x|^2 +  \Biggl| \sum_{i=k+1}^{n} (\phi _{n+i}(y) -  \phi _{n+i}(x)) ( \phi _{i}(x)- \phi _{i}(y))  \Biggl| \\
  & \leq L _{2n+1}|y-x|^2 + \sum_{i=k+1}^{n} L_{n+i} |y-x| \,
  \min\{1,L_i\}
  |p-q|\\
 & \leq   \left(L_{2n+1} + \sum_{i=k+1}^n L_{n+i}\min\{1,L_i\} \right)
 |q-p|^2.
\end{align*}
Hence  condition \eqref{i:d2} in Definition \ref{d:tame} holds,
and we have shown that $f$ is $(L_{n+1}'',\ldots,L_{2n+1}'')$-tame
with the constants defined in \eqref{eq:L''}. By Theorem
\ref{propK=N} applied to $f$, it then follows that there exists an
extension $\bar f=(\bar f_{n+1}, \dots , \bar f_{2n+1}):\R^n \to
\R^{n+1}$ of $f$ which is $(L_{n+1}''',\ldots,L_{2n+1}''')$-tame
with
\begin{align}\label{eq:Li'''}
L_{i}'''&= 2 C(n) \max \left\{|(L_{n+1}'',\ldots,L_{2n}'')|,L_{2n+1}''\right\}\notag\\
&= 2 C(n) \max
\left\{|(L_{n+1},\ldots,L_{2n})|,L_{2n+1}+\sum_{i=k+1}^n L_{n+i}
\min\{1,L_i\}\right\}.
\end{align}

To construct an extension $\overline{\phi}$ for the given map
$\phi$, we first extend independently the components
$\phi_{k+1},\ldots,\phi_n$. For $i=k+1,\ldots, n$, we simply apply
McShane's extension theorem to extend the $L_i$-Lipschitz function
$\phi_i:E \to \mathbb{R}$ to an $L_i$-Lipschitz function
$\overline{\phi}_i:\mathbb{R}^k \to \mathbb{R}$. With the
extensions $\overline{f}$ and
$\overline{\phi}_{k+1},\ldots,\overline{\phi}_n$ at hand, we are
now able to prove that the map $\bar \phi : \R^k \to \R^{2n+1-k}$
defined  in \eqref{goodextension} is the desired tame extension of
$\phi$.

Recalling that
$\overline{f}|_{\Gamma^{\phi_{k+1},\ldots,\phi_n}(E)}=f$ and
keeping in mind expression \eqref{definf} for $f$, it is clear
that $\bar \phi$ is an extension of $\phi$. Moreover, for every
$i=n+1,\ldots,2n$, the function
\begin{displaymath}
\bar \phi_i:\mathbb{R}^k \to \mathbb{R},\quad \bar \phi_i(x)=\bar
f_i(x,\bar \phi_{k+1}(x),\ldots,\bar \phi_n(x))
\end{displaymath}
 is Lipschitz:
\begin{align*}
|\overline{\phi}_i(x)-\overline{\phi}_i(y)|\leq
\mathrm{Lip}(\overline{f}_i)\Big(1+\sum_{j=k+1}^n
L_j^2\Big)^{\frac{1}{2}}|x-y|.
\end{align*}
Recalling formula \eqref{eq:Li'''} for
$L_i'''=\mathrm{Lip}(\overline{f}_i)$, we conclude that
$\overline{\phi}_i$ is $L_i'$-Lipschitz with
\begin{displaymath}L_i'= c_n\Big(1+\sum_{j=k+1}^n L_j^2\Big)^{\frac{1}{2}}\max  \left\{|(L_{n+1},\ldots,L_{2n})|,L_{2n+1}+\sum_{i=k+1}^n
L_{n+i}\min\{1,L_i\}\right\}
\end{displaymath}
for $i=n+1,\ldots,2n$ and a constant $c_n$ that depends
only on $n$.

As a consequence, the only non-trivial condition to check  for the
map $\bar \phi$ is the second part of the tameness condition,
namely \eqref{i:d2} in Definition \ref{d:tame}. Let $x,y \in
\R^k,$ and, for simplicity, put $p:=(x,\bar \phi _{k+1}(x), \dots
,\bar \phi _n (x)), q:=(y,\bar \phi _{k+1}(y), \dots , \bar \phi
_n (y))$ and
\begin{displaymath}
\bar \psi(y):=(\bar \phi_{n+1}(y),\ldots,  \bar \phi _{n+k}
(y))=(\overline{f}_{n+1}(q),\ldots,\overline{f}_{n+k}(q)).
\end{displaymath}
We have that
 \begin{equation*}
\begin{aligned}
J(x,y):=& \left| \bar \phi_{2n+1}(y) - \bar \phi_{2n+1}(x)
-\langle \bar \psi (y),
y-x\rangle -\tfrac{ 1}{ 2} \sum_{i=k+1}^n \bar \phi_i(y)\bar \phi_{n+i} (x) -\bar \phi_i(x)\bar \phi_{n+i} (y)  \right|\\
 & \overset{\eqref{goodextension}}{\leq} \left| \bar f _{2n+1} (q) - \bar f _{2n+1} (p)-\left\langle
 (
\bar f_{n+1}(q), \ldots , \bar f_{2n} (q) ), ( q_1-p_1, \ldots ,
q_{n} -p_{n} )\right\rangle
 \right| \\
 & \quad + \tfrac{ 1}{ 2} \Biggl|  \sum_{i=k+1}^{n} \bar \phi_{i}(y)( \bar f_{n+i} (q) - \bar f _{n+i}(p)) -
  \bar \phi_{i}(x)( \bar f_{n+i} (q) - \bar f _{n+i}(p))
  \Biggl|\\
 & \leq \left| \bar f _{2n+1} (q) - \bar f _{2n+1} (p)-\left\langle
 (
\bar f_{n+1}(q), \ldots , \bar f_{2n} (q) ), ( q_1-p_1 , \ldots ,
q_{n} -p_{n} )\right\rangle
 \right| \\
 & \quad + \tfrac{ 1}{ 2} \sum_{i=k+1}^{n} |\bar \phi_{i}(y)-  \bar \phi_{i}(x)| | \bar f_{n+i} (q) - \bar f
 _{n+i}(p)|.
\end{aligned}
\end{equation*}
Summing this term and the corresponding expression with the roles
of $x$ and $y$ reverted, we find by the tameness condition of
$\overline{f}$ and the Lipschitz continuity of $\bar
\phi_{k+1},\ldots,\bar \phi_n$ that
\begin{align*}
J(x,y)+J(y,x)&\leq L_{2n+1}''' |p-q|^2 + \sum_{i=k+1}^n L_i
L_{n+i}''' |x-y|
|p-q|\\
&\leq \Big(L_{2n+1}''' \Big(1+\sum_{j=k+1}^n L_j^2\Big)+
\sum_{i=k+1}^n L_i L_{n+i}''' \Big(1+\sum_{j=k+1}^n
L_j^2\Big)^{\frac{1}{2}}\Big) |x-y|^2,
\end{align*}
for arbitrary $x,y\in \mathbb{R}^k$. Now we only have to recall
the expressions for $L_j'''$, $j=n+1,\ldots,2n$, given in
\eqref{eq:Li'''}. Then, for a constant $c_n$ that depends only on
$n$, the last tameness constant of $\overline{\phi}$ can be chosen
as
\begin{displaymath}
L_{2n+1}'=  c_n\Big(1+\sum_{j=k+1}^n L_j^2\Big)\max
\left\{|(L_{n+1},\ldots,L_{2n})|,L_{2n+1}+\sum_{i=k+1}^n
L_{n+i}\min\{1,L_i\}\right\}.
\end{displaymath}
\end{proof}

\subsection{Extension result for low-dimensional intrinsic
Lipschitz graphs}

Combining the previous results, we establish the extension theorem
for low-dimensional intrinsic Lipschitz graphs in $\mathbb{H}^n$,
Theorem \ref{t:MainIntro} from the introduction.

\begin{proof}[Proof of Theorem \ref{t:MainIntro}] Let $1\leq k
\leq n$. First, if $\phi=(\phi_{k+1},\ldots,\phi_{2n+1})$ is
intrinsic $L$-Lipschitz on $E\subset \mathbb{R}^k$, Proposition
\ref{p:FromIntrLipToTameGENERAL} implies that
$(\phi_{k+1},\ldots,\phi_{2n},-\phi_{2n+1})$ is
$(L_{k+1},\ldots,L_{2n+1})$-tame with
\begin{displaymath}
L_i=L\quad\text{for }i\neq 2n+1,\quad\text{and}\quad L_{2n+1}= 2L^2.
\end{displaymath}
Applying the extension result from Theorem \ref{propK=N}, if
$k=n$, and Theorem \ref{propKminoreN}, if $k<n$, to this tame map
yields an $(L_{k+1}',\ldots,L_{2n+1}')$-tame extension
$(\overline{\phi}_{k+1},\ldots,\overline{\phi}_{2n},-\overline{\phi}_{2n+1})$,
where the tameness constants depend only on
$L_{k+1},\ldots,L_{2n+1}$ (thus on $L$), $k$, and $n$. Finally, we
use Proposition \ref{p:FromTameToIntrLip} to conclude that $
\overline{\phi}:=
(\overline{\phi}_{k+1},\ldots,\overline{\phi}_{2n},\overline{\phi}_{2n+1})
$ is an intrinsic $L'$-Lipschitz function on $\V$ with
\begin{displaymath}
L':=\max\{|(L_{k+1}',\ldots,L_{2n}')|,\sqrt{L_{2n+1}'}\}.
\end{displaymath}

The better quantitative control over the intrinsic Lipschitz
constant if $k=n=1$ follows since Proposition
\ref{p:FromIntrLipToTameGENERAL} yields $ L_2=\min\{L,2L^2\} $ in
this case. Then $L_2'$ and $L_3'$ in Theorem  \ref{propK=N} can be
bounded from above by a constant times $L^2$, and it follows that
we can take $L'= C\max\{L^2,L\}$ for a suitable constant $C$.
\end{proof}

\section{Corona decomposition for $1$-dimensional intrinsic
Lipschitz graphs}\label{s:corona}

The main result of this section is a corona decomposition of
$1$-dimensional iLG in $\mathbb{H}^n$, $n>1$, by iLG with smaller
Lipschitz constant (Theorem \ref{Theorem 3.15n>1}). The
corresponding result for $n=1$ (and tame maps) was proven in
\cite[Theorem 3.15]{fssler2019singular}, motivated by an
application to singular integral operators on $1$-dimensional iLG
in $\mathbb{H}^1$. As was the case for \cite{fssler2019singular},
our argument is ultimately based on a corona decomposition for
Euclidean Lipschitz graphs. The version that we will employ looks
a little different from the formulations in the literature, so we
state it in Section \ref{ss:coronaEucl} and explain how to deduce
it from the ``standard'' corona decomposition for Euclidean
Lipschitz graphs given in \cite[p.57, Definition 3.19 and p.61,
(3.33)]{MR1251061}. Based on these preparations, we prove the
result for iLG in Section \ref{s:CoronaiLG}.

\subsection{Corona decomposition for $1$-dimensional Euclidean
Lipschitz graphs}\label{ss:coronaEucl}
%
%

\begin{definition}[Dyadic intervals and trees]\label{Dyadic intervals and trees}
The family of standard \emph{dyadic intervals} of $\R$ is called
``$\calD$''. For $j \in \Z$, we  write $\calD_{j} \subset \calD$
for the dyadic intervals $Q$ of length $|Q|=2^{-j}$. A collection
$\calT \subset \calD$ is  a \emph{tree} if
\begin{itemize}
\item[(T1)] $\calT$ contains a \emph{top interval} $Q(\calT)$,
that is, a unique maximal element. \item[(T2)] $\calT$ is
\emph{coherent}: if $Q \in \calT$, then $Q' \in \calT$ for all
dyadic intervals $Q \subset Q' \subset Q(\calT)$. \item[(T3)] If
$Q \in \calT$, then either both, or neither, of the children of
$Q$ lie in $\calT$.
\end{itemize}
\end{definition}

\begin{definition}[Coronization]\label{d:coronization} A decomposition
$\mathcal{D} = \mathcal{G} \dot{\cup} \mathcal{B}$ of
$\mathcal{D}$ into \emph{good} intervals $\mathcal{G}$ and
\emph{bad} intervals $\mathcal{B}$ (with $\mathcal{G}\cap
\mathcal{B}=\emptyset$) is called a \emph{coronization} if there
exists a constant $C$ such that the following conditions are
satisfied:
 \begin{enumerate}
\item The intervals in $\mathcal{B}$ satisfy a Carleson packing
condition:
\begin{equation*}
\sum_{\substack{ Q \in \mathcal{B} \\ Q \subset Q_0} } |Q| \leq C
|Q_0|, \quad \text{for all } Q_0 \in \mathcal{D}.
\end{equation*}
\item The intervals in $\mathcal{G}$ can be decomposed into a
\emph{forest} $\mathcal{F}$ of disjoint trees $\mathcal{T}$
\begin{equation*}
\mathcal{G} = \dot{\bigcup}_{\mathcal{T} \in \mathcal{F}}
\mathcal{T}
\end{equation*}
whose top intervals satisfy a Carleson packing condition:
\begin{equation*}
\sum_{\substack{ \mathcal{T} \in \mathcal{F} \\ Q(\mathcal{T})
\subset Q_0 } } |Q(\mathcal{T})| \leq C |Q_0|, \quad \text{for all
} Q_0 \in \mathcal{D}.
\end{equation*}
\end{enumerate}
\end{definition}

\begin{thm}[Corona decomposition for
Lipschitz maps]\label{t:DSAdaptation} For every $n\in \mathbb{N}$
and $\delta \in (0,1)$, there exists a constant $C \geq 1$ with
the following property. Let $\psi \colon \mathbb{R} \to
\mathbb{R}^{2n-1}$ be $1$-Lipschitz.
Then, there exists a
coronization $\mathcal{D} = \mathcal{B} \dot{\cup} \mathcal{G}$ and a forest $ \mathcal{F}$,
satisfying the conditions in Definition \ref{d:coronization} with
constant $C$, such that the following holds. For every
$\mathcal{T} \in \mathcal{F}$ there is a $2$-Lipschitz linear
function $L_{\mathcal{T}} \colon \mathbb{R} \to \mathbb{R}^{2n-1}$
and a $\delta$-Lipschitz function $\psi_{\mathcal{T}} \colon
\mathbb{R} \to \mathbb{R}^{2n-1}$ such that
\begin{equation}\label{form15Lipschitz}
|\psi(s) - (\psi_{\mathcal{T}} +
L_{\mathcal{T}})(s)|
\leq \delta |Q|, \qquad s
\in 2Q, \; Q \in \mathcal{T},
\end{equation}
where $2Q$ is the interval with the same midpoint as $Q$ but twice
its length.
\end{thm}

\begin{proof}
For $n=1$, the result was deduced in \cite[Theorem
3.20]{fssler2019singular} from the corona decomposition in
\cite[p.61, (3.33)]{MR1251061}.  The same approach works for
$n>1$, although the reduction to \cite[p.61, (3.33)]{MR1251061} is
now a bit more involved.

\medskip

To start the proof, let us fix $n>1$. By  \cite[p.61, (3.33) and
p.328, \S 2.2]{MR1251061} we know that for every $\delta'\in
(0,1)$, there exists a constant $C=C(n,\delta')$ such that, for
every $1$-Lipschitz function $\psi \colon \mathbb{R} \to
\mathbb{R}^{2n-1}$, there is a coronization $\mathcal{D} =
\mathcal{B} \dot{\cup} \mathcal{G}$ with constant $C$ that
satisfies the following property. For every tree $\mathcal{T}$ in
the associated forest  $\mathcal{F}$  there is a $1$-dimensional
$\delta'$-Lipschitz graph $\Gamma_{\mathcal{T}}$ such that
\begin{equation}\label{eq:DS_approx}
\mathrm{dist}\left((s,\psi(s)),\Gamma_{\mathcal{T}}\right)\leq
\delta' |Q|,\quad \text{for all }s\in 2Q,\, Q\in \mathcal{T}.
\end{equation}
In other words, there exists a $\delta'$-Lipschitz function
$\varphi_{\mathcal{T}}:\mathbb{R} \to \mathbb{R}^{2n-1}$ and $R\in
O(2n)$ such that \eqref{eq:DS_approx} holds for
\begin{equation}\label{eq:RotatedGraph}
\Gamma_{\mathcal{T}} = \{R(x,\varphi_{\mathcal{T}}(x)):\; x\in
\mathbb{R}\}.
\end{equation}
To be precise, \cite[p.61, (3.33)]{MR1251061} refers to a system
of dyadic cubes on the graph of $\psi$, rather than in the domain
$\mathbb{R}$, but \eqref{eq:DS_approx} is easily deduced. The thus
given coronization is the same that appears in the statement that
we are about to prove, so the only challenge is to find
$\delta'(\delta,n)$ so that we can deduce from
\eqref{eq:DS_approx} that \eqref{form15Lipschitz} holds for
suitable $\psi_{\mathcal{T}}$ and $L_{\mathcal{T}}$. As it will be
convenient to work in coordinates, we represent $R$ as a matrix
with respect to the standard basis of $\mathbb{R}^{2n}$,
\begin{equation}\label{eq:RMatr}
R=
\begin{pmatrix}b_{1,1}&\cdots&b_{1,2n}\\&\ddots&\\b_{2n,1}&\cdots&b_{2n,2n}\end{pmatrix},
\end{equation}
so that for
$\varphi_{\mathcal{T}}=(\varphi_{\mathcal{T},2},\ldots,\varphi_{\mathcal{T},2n})$
the identity \eqref{eq:RotatedGraph} reads
\begin{equation}\label{eq:Gamma_T_coords}
\Gamma_{\mathcal{T}} = \left\{\begin{pmatrix}b_{1,1}x +
\sum_{l=2}^{2n} \varphi_{\mathcal{T},l}(x) b_{1,l}\\\vdots\\
b_{2n,1}x + \sum_{l=2}^{2n} \varphi_{\mathcal{T},l}(x) b_{2n,l}
\end{pmatrix}:\; x\in \mathbb{R}\right\}.
\end{equation}
Then the approximation property \eqref{eq:DS_approx} means exactly
that for every $Q\in \mathcal{T}$, and for every $s\in 2Q$, there
exists $x_s\in \mathbb{R}$ such that
\begin{equation}\label{eq:DS_approx2}
\left|\begin{pmatrix}s\\\psi_2(s)\\\vdots\\\psi_{2n}(s)\end{pmatrix}
-
\begin{pmatrix}
b_{1,1}x_s +
\sum_{l=2}^{2n} \varphi_{\mathcal{T},l}(x_s) b_{1,l}\\
b_{2,1}x_s +
\sum_{l=2}^{2n} \varphi_{\mathcal{T},l}(x_s) b_{2,l}\\\vdots\\
b_{2n,1}x_s + \sum_{l=2}^{2n} \varphi_{\mathcal{T},l}(x_s)
b_{2n,l}
\end{pmatrix} \right| \leq \delta' |Q|.
\end{equation}
For a given point $s$, there can be several points with this
property, but we just choose one of them and call it $x_s$. We
will apply \eqref{eq:DS_approx2} for small enough $\delta'\in
(0,1)$ depending on $n$ and the parameter $\delta$ in the
statement of the theorem. The precise condition will appear in
\eqref{eq:delta'} below, but the bound on $\delta'$ has to be
chosen such that $\Gamma_{\mathcal{T}}$ in
\eqref{eq:Gamma_T_coords} can be written as graph \emph{over the
$x_1$-axis} of a function of the form
$\psi_{\mathcal{T}}+L_{\mathcal{T}}$, where
$\psi_{\mathcal{T}}:\mathbb{R} \to \mathbb{R}^{2n-1}$ is
$\delta$-Lipschitz, and $L_{\mathcal{T}}:\mathbb{R}\to
\mathbb{R}^{2n-1}$ is linear with Lipschitz constant $2$. For this
purpose, we define
\begin{displaymath}
z(x):= b_{1,1}x + \sum_{l=2}^{2n} \varphi_{\mathcal{T},l}(x)
b_{1,l}
\end{displaymath}
and, recalling \eqref{eq:Gamma_T_coords}, our goal is to write
\begin{equation}\label{eq:Goal_z(x)}
\begin{pmatrix}b_{1,1}x +
\sum_{l=2}^{2n} \varphi_{\mathcal{T},l}(x) b_{1,l}\\
b_{2,1}x +
\sum_{l=2}^{2n} \varphi_{\mathcal{T},l}(x) b_{2,l}\\\vdots\\
b_{2n,1}x + \sum_{l=2}^{2n} \varphi_{\mathcal{T},l}(x) b_{2n,l}
\end{pmatrix} = \begin{pmatrix} z(x) \\ \psi_{\mathcal{T},2}(z(x))
+ L_{\mathcal{T},2}(z(x))\\ \vdots\\ \psi_{\mathcal{T},2n}(z(x)) +
L_{\mathcal{T},2n}(z(x))\end{pmatrix},\quad x\in \mathbb{R}
\end{equation}
for $\psi_{\mathcal{T}}$ and $L_{\mathcal{T}}$ as mentioned above.
Assume for a moment that we know that $b_{1,1}\neq 0$. Then we can
define
\begin{equation}\label{eq:Def_L_T}
L_{\mathcal{T},l}(x):= \frac{b_{l,1}}{b_{1,1}}x,\quad \text{for
}l=2,\ldots,2n
\end{equation}
and solving \eqref{eq:Goal_z(x)} for $\psi_{\mathcal{T}}(z(x))$
yields
\begin{equation}\label{eq:Psi_T}
\psi_{\mathcal{T},l}(z(x))= \sum_{i=2}^{2n}
\left(b_{l,i}-\tfrac{b_{l,1}}{b_{1,1}}b_{1,i}\right)\varphi_{\mathcal{T},i}(x),\quad\text{for
}l=2,\ldots,2n.
\end{equation}
To establish that the thus defined functions $L_{\mathcal{T}}$ and
$\psi_{\mathcal{T}}$ are Lipschitz with the claimed constants will
require us to prove a suitable uniform upper bound for
$|(b_{2,1},\ldots,b_{2n,1})|/|b_{1,1}|$ assuming an upper bound
for $\delta'$ (this will show in particular that $b_{1,1}\neq 0$).
Supposing for a moment that this can be done, we will then prove
that  $z:\mathbb{R} \to \mathbb{R}$ is a bi-Lipschitz
homeomorphism  if $\delta'$ is chosen small enough (it is
obviously always Lipschitz continuous). This will finally also
yield that $\Gamma_{\mathcal{T}}$ is indeed the graph of
$\psi_{\mathcal{T}}+L_{\mathcal{T}}$. Before entering the
computations, we recall that $R$ is an orthogonal matrix, so all
its rows and columns have length $1$, and in particular we obtain
that
\begin{equation}\label{eq:column_length}
1\geq \left|(b_{1,2},\cdots,b_{1,2n})\right|.
\end{equation}
Combined with the $\delta'$-Lipschitz continuity of
$\varphi_{\mathcal{T}}$, this yields that
\begin{align}\label{eq:z(x)}
|z(x_1)-z(x_2)|&= \left|b_{1,1}(x_1-x_2)+\sum_{l=2}^{2n}
b_{1,l}(\varphi_{\mathcal{T},l}(x_1)-\varphi_{\mathcal{T},l}(x_2))
\right|\notag \\& \geq \left(|b_{1,1}|-\delta'
\left|(b_{1,2},\ldots,b_{1,2n})\right|\right)|x_1-x_2|\notag\\
&\overset{\eqref{eq:column_length}}{\geq}
\left(|b_{1,1}|-\delta'\right)|x_1-x_2|
\end{align}
for all $x_1,x_2\in \mathbb{R}$. Once again, a universal positive
lower bound on $b_{1,1}$ will conclude the proof if $\delta'$ is
chosen small enough.

\medskip

Motivated by these considerations, we now concentrate our efforts
on proving that
\begin{equation}\label{eq:b_goal}
\delta' \leq \tfrac{1}{100 \sqrt{2n-1}}\quad \Rightarrow \quad
|b_{1,1}|\geq 1/\sqrt{5}\quad \left(\Leftrightarrow\quad
|(b_{2,1},\ldots,b_{2n,1})|/|b_{1,1}| <2.\right)
\end{equation}
To see why such a statement is plausible, it may help the reader
to picture the case $n=1$. Then \eqref{eq:b_goal} essentially says
that if the graph of the $1$-Lipschitz function $\psi$ is
well-approximated by $\Gamma_{\mathcal{T}}$, which is obtained by
rotating the graph of a $\delta'$-Lipschitz function
$\varphi_{\mathcal{T}}$ over the $x_1$-axis by angle $\theta$,
then $|\tan(\theta)|=|b_{2,1}/b_{1,1}|<2$ if $\delta'$ is small
enough.

We will now prove \eqref{eq:b_goal} for $n>1$. To this end, pick
an arbitrary interval $S\in \mathcal{T}$ and denote by $s_1$ and
$s_2$ the endpoints of $S$. Moreover, let $x:=x_{s_1}\in
\mathbb{R}$ be such that  \eqref{eq:DS_approx2} holds for $s=s_1$
and $Q=S$, and in the same way, associate to the other endpoint
$s_2$ a point $x':=x_{s_2}\in \mathbb{R}$ such that
\eqref{eq:DS_approx2} holds for $s=s_2$ and $Q=S$. Using this
property, combined with the $1$-Lipschitz continuity of $\psi$,
the matrix representation \eqref{eq:RMatr} of $R\in O(2n)$, and
the $\delta'$-Lipschitz continuity of $\varphi_{\mathcal{T}}$, we
conclude that
\begin{align*}
\Bigg|
\begin{pmatrix}b_{2,1}\\\vdots\\
b_{2n,1}\end{pmatrix} (x-x')&+
\begin{pmatrix}\langle(b_{2,2},\ldots,b_{2,2n}),\varphi_{\mathcal{T}}(x)
-\varphi_{\mathcal{T}}(x')\rangle\\\vdots\\\langle(b_{2n,2},\ldots,b_{2n,2n}),\varphi_{\mathcal{T}}(x)
-\varphi_{\mathcal{T}}(x')\rangle\end{pmatrix}
\Bigg|\\
& \overset{\eqref{eq:DS_approx2}}{\leq}
|\psi(s_1)-\psi(s_2)|+ 2
\delta' |S|\leq |s_1-s_2| + 2 \delta' |S|\\
&\overset{\eqref{eq:DS_approx2}}{\leq}
\left|b_{1,1}(x-x')+\sum_{l=2}^{2n}b_{1,l}(\varphi_{\mathcal{T},l}(x)-\varphi_{\mathcal{T},l}(x'))\right|
+4\delta'|S|\\
&\leq
|b_{1,1}||x-x'|+|\varphi_{\mathcal{T}}(x)-\varphi_{\mathcal{T}}(x')|+
4\delta'|S|\leq \left(|b_{1,1}|+\delta' \right)|x-x'|+ 4\delta'
|S|.
\end{align*}
This implies that
\begin{align}\label{eq:PageA}
\left|\begin{pmatrix}b_{2,1}\\\vdots\\b_{2n,1}\end{pmatrix}\right|
|x-x'|\leq
&\left(|b_{1,1}|+(\sqrt{2n-1}+1)\delta' \right)|x-x'|+
4\delta' |S|.
\end{align}
The above estimates  hold for arbitrary points in $2S$, but the
fact that $s_1$ and $s_2$ are the endpoints of $S$, allow us to
show that $|S|\lesssim |x-x'|$. More precisely, since
$\varphi_{\mathcal{T}}$ is $\delta'$-Lipschitz, we find
\begin{align*}
|S|=|s_1-s_2|\leq & \left|s_1-b_{1,1}x-\sum_{l=2}^{2n}
b_{1,l}\varphi_{\mathcal{T},l}(x)\right|+
\left|s_2-b_{1,1}x'-\sum_{l=2}^{2n} b_{1,l}
\varphi_{\mathcal{T},l}(x')\right|\\
&+ |b_{1,1}| |x-x'| + \delta' |x-x'|\\
\overset{\eqref{eq:DS_approx2}}{\leq}& 2\delta'|S| +
|b_{1,1}||x-x'| + \delta'|x-x'|.
\end{align*}
Moving the terms with $|S|$ to the left-hand side, we conclude for
small enough $\delta'$ as in
 \eqref{eq:b_goal}  that
\begin{displaymath}
|S|\leq
\left(\frac{|b_{1,1}|+\delta'}{1-2\delta'}\right)|x-x'|\leq
\frac{5}{2}|x-x'|.
 \end{displaymath}
 Inserting this estimate in \eqref{eq:PageA} and dividing both sides by $|x-x'|$, we
 observe that
 \begin{displaymath}
 |(b_{2,1},\ldots,b_{2n,1})|\leq |b_{1,1}| + \delta' \left(11+
 \sqrt{2n-1}\right).
 \end{displaymath}
As the columns of the orthogonal matrix $R$ are unit vectors, we
have
\begin{equation}\label{eq:b_11_bound}
1-b_{1,1}^2 =|(b_{2,1},\ldots,b_{2n,1})|^2\leq b_{1,1}^2 +
2\left(11+\sqrt{2n-1}\right) \delta' +
\left(11+\sqrt{2n-1}\right)^2 {\delta'}^2.
\end{equation}
If $\delta'$ is small enough, say as in \eqref{eq:b_goal}, then we
can deduce from \eqref{eq:b_11_bound} that $b_{1,1}^2 > 1/5$, and
hence
\begin{equation}\label{eq:b_ratio_estimate}
\frac{\left|b_{2,1},\ldots,b_{2n,1}\right|^2}{b_{1,1}^2} =
\frac{1-b_{1,1}^2}{b_{1,1}^2} < 4,
\end{equation}
which concludes the proof of \eqref{eq:b_goal}.

\medskip

It is now immediate from the formula \eqref{eq:Def_L_T} that
$L_{\mathcal{T}}$ is $2$-Lipschitz. Moreover, it follows from the
formula for $\psi_{\mathcal{T}}$ in \eqref{eq:Psi_T}, the
$\delta'$-Lipschitz continuity of $\varphi_{\mathcal{T}}$, the
Cauchy-Schwarz inequality, and \eqref{eq:b_ratio_estimate}
 that
\begin{displaymath}
|\psi_{\mathcal{T}}(z(x_1))-\psi_{\mathcal{T}}(z(x_2))|\leq 3
\sqrt{2n-1}
|\varphi_{\mathcal{T}}(x_1)-\varphi_{\mathcal{T}}(x_2)|\leq
3\sqrt{2n-1}\,\delta'\, |x_1-x_2|
\end{displaymath}
for all $x_1,x_2\in \mathbb{R}$. Since $|b_{1,1}|\geq 1/\sqrt{5}$
and \eqref{eq:z(x)} holds, it is clear that for
\begin{equation}\label{eq:delta'}
\delta' \leq \frac{\delta}{100\sqrt{2n-1}}
\end{equation}
the function $\psi_{\mathcal{T}}$ is $\delta$-Lipschitz, as
required.

\medskip

It remains to verify the approximation condition
\eqref{form15Lipschitz}. This follows easily from
\eqref{eq:DS_approx2}, which can now be rewritten as
\begin{equation}\label{eq:finalEst}
\left|\begin{pmatrix}s\\\psi(s)\end{pmatrix}-\begin{pmatrix}z(x_s)\\
\psi_{\mathcal{T}}(z(x_s))+L_{\mathcal{T}}(z(x_s))\end{pmatrix}\right|\leq
\delta' |Q|,\quad Q\in \mathcal{T},\, s\in 2Q.
\end{equation}
Since $|s-z(x_s)|\leq \delta' |Q|$, and $\psi_{\mathcal{T}} +
L_{\mathcal{T}}$ is $3$-Lipschitz, it follows that
\eqref{eq:finalEst} holds with $z(x_s)$ replaced by $s$, and
$\delta'$ replaced by $\delta$, recalling the bound
\eqref{eq:delta'}. This concludes the proof of  Theorem
\ref{t:DSAdaptation}.
\end{proof}

We make one more modification in the construction of
$\psi_{\mathcal{T}}$, see Corollary \ref{c:DSAdaptBdry} below.
This is an additional step compared to the proof for $n=1$ in
\cite{fssler2019singular}, necessitated by the noncommutativity of
codimension-$1$ vertical subgroups in $\mathbb{H}^n$ for $n>1$.

In the setting of Theorem \ref{t:DSAdaptation}, if $\mathcal{T}
\in \mathcal{F}$ is a tree with top interval $Q(\mathcal{T})$, we
denote by
 $\mathcal{S} (\mathcal{T})$  the (possibly empty) collection of
minimal intervals in $\mathcal{T}$. Moreover, we write
\begin{equation*}
E :=  Q(\mathcal{T})\setminus  \bigcup _{ S \in \mathcal{S}
(\mathcal{T})} S
\end{equation*}
for the set of points in $Q(\mathcal{T})$ in infinite branches of
$\mathcal{T}$. By the approximation condition \eqref{eq:DS_approx}
in the corona decomposition, we have
\begin{equation}\label{eq:equalOnE}
\psi (s) = [{L}_{\mathcal{T}} + \psi_{\mathcal{T}}] (s), \quad
\text{for all }s \in E.
\end{equation}
For the application to intrinsic Lipschitz graphs  in
$\mathbb{H}^n$, $n>1$, in Section \ref{s:CoronaiLG}, it will be
beneficial to have the identity \eqref{eq:equalOnE} also in all
points $s\in \bigcup_{S\in \mathcal{S}(\mathcal{T})}\partial S$.
Otherwise error terms will appear depending on the position of the
intrinsic graph and caused by the non-commutativity of
codimension-$1$ vertical subgroups.

\begin{cor}\label{c:DSAdaptBdry}
For $\mathcal{T}\in \mathcal{F}$, the function $\psi_{\mathcal{T}}$
in Theorem  \ref{t:DSAdaptation} can be constructed so that
\begin{equation}\label{eq:equalOnpartial S}
\psi (s) = [{L}_{\mathcal{T}} + \psi_{\mathcal{T}}] (s), \quad
\text{for all }s \in \bigcup_{S\in\mathcal{S}(\mathcal{T})}
\partial S.
\end{equation}
\end{cor}

\begin{proof}
Fix an arbitrary tree $\mathcal{T}$ in the forest $\mathcal{F}$
associated to the corona decomposition given by Theorem
\ref{t:DSAdaptation} for the $1$-Lipschitz function
$\psi:\mathbb{R} \to \mathbb{R}^{2n-1}$ and parameter
``$\delta/(8\sqrt{2n-1})$''. Assume that
$\mathcal{S}(\mathcal{T})$ is nonempty, otherwise there is nothing
to prove. Let $\psi_{\mathcal{T}}$ and $L_{\mathcal{T}}$ be the
associated functions provided by Theorem \ref{t:DSAdaptation}. We
will now slightly modify $\psi_{\mathcal{T}}$ inside each minimal
interval $S\in \mathcal{S}(\mathcal{T})$ so that
\eqref{eq:equalOnpartial S} holds for the modified function
$\widetilde{\psi}_{\mathcal{T}}$, which will be $\delta$-Lipschitz
and satisfy \eqref{eq:DS_approx} (for $\psi, L_{\mathcal{T}}$ and
``$\delta$''). Set $\widetilde{\psi}_{\mathcal{T}} (s):=
\psi_{\mathcal{T}} (s) = \psi (s)-L_{\mathcal{T}} (s)$ for $s\in
E.$ It is possible to define the modified function
$\psi_{\mathcal{T}}$ on $Q(\mathcal{T})\setminus E$ by considering
one minimal interval $S$ at the time, since the boundary points of
$S$ either belong to $E$, where \eqref{eq:equalOnE} already holds,
or they are boundary points of two adjacent minimal intervals, so
that the modification will be well-defined. Assuming that
$\psi_{\mathcal{T}}$ is defined on the entire real line, we set
$\widetilde{\psi}_{\mathcal{T}}=\psi_{\mathcal{T}}$ outside
$Q(\mathcal{T})$.

\medskip

Fix $S=[a,b] \in \mathcal{S}(\mathcal{T})$. The components of
$\widetilde{\psi}_{\mathcal{T},l}$, $l=2,\ldots,2n$, can be taken
of the form
\begin{displaymath}
(\widetilde{\psi}_{\mathcal{T},l})|_S (s):=
\psi_{\mathcal{T},l}(s) + (\psi_l(a)-L_{\mathcal{T},l}(a)
-\psi_{\mathcal{T},l}(a))+ c_{S,l} (s-a),\quad s\in S,
\end{displaymath}
for suitable constants $c_{S,l}\in \mathbb{R}$ with
\begin{equation}\label{eq:CondOnC}
|(c_{S,2},\ldots,c_{S,2n})|\leq \delta/4. \end{equation} Since we
have merely added an affine function, the modified map
$\widetilde{\psi}_{\mathcal{T}}$ is clearly Lipschitz with
constant $\delta/(8\sqrt{2n-1}) + |(c_{S,2},\ldots,c_{S,2n})|\leq
(3/8)\delta$.

\medskip

Moreover,
$$\widetilde{\psi}_{\mathcal{T},l}(a)=\psi_l(a) -L_{\mathcal{T},l}(a),$$ by
construction, and
\begin{equation}\label{displaymath1.0}
\widetilde{\psi}_{\mathcal{T},l}(s) +L_{\mathcal{T},l}(s)
=\psi_{\mathcal{T},l}(s) +L_{\mathcal{T},l}(s)=\psi_l(s),\quad
\text{for all }s\in E.
\end{equation}
So we only have to show that $c_{S,2},\ldots,c_{S,2n}$ can be
chosen such that the two functions match also at the other
endpoint of $S$, that is
$\widetilde{\psi}_{\mathcal{T},l}(b)=\psi_l(b) -
L_{\mathcal{T},l}(b)$ for $l=2,\ldots,2n$. This forces us to take
\begin{displaymath}
c_{S,l}:= \frac{1}{b-a}\left[\psi_l(b) -L_{\mathcal{T},l}(b) -
\psi_{\mathcal{T},l}(b)
-(\psi_l(a)-L_{\mathcal{T},l}(a)-\psi_{\mathcal{T},l}(a))\right].
\end{displaymath}
Then the approximation condition \eqref{eq:DS_approx} provides the
desired bound for $c_{S,l}$, $l=2,\ldots,2n$, namely
\begin{align*}
|c_{S,l}|&\leq \frac{1}{b-a}|\psi_l(b)-L_{\mathcal{T},l}(b)-
\psi_{\mathcal{T},l}(b)|+
\frac{1}{b-a}|\psi_l(a)-L_{\mathcal{T},l}(a)-\psi_{\mathcal{T},l}(a)|\\
&\leq \frac{\delta |S|}{8\sqrt{2n-1}|S|} + \frac{\delta
|S|}{8\sqrt{2n-1}|S|}= \frac{\delta}{4\sqrt{2n-1}},
\end{align*}
which yields \eqref{eq:CondOnC}. We apply the same procedure for
every minimal interval $S$. This yields the modified function
$\widetilde{\psi}_{\mathcal{T}}:\mathbb{R} \to \mathbb{R}^{2n-1}$,
which is $\frac 3 8 \delta $-Lipschitz. This follows from
\begin{displaymath}
\widetilde{\psi}_{\mathcal{T}}(s)=\psi
(s)-L_{\mathcal{T}}(s)\text{ for } s\in E \cup  \bigcup_{S\in
\mathcal{S}(\mathcal{T})}\partial S \quad \text{and}\quad
\widetilde{\psi}_{\mathcal{T}}(s)=\psi_{\mathcal{T}}(s)\text{ for
}s\in E,
\end{displaymath}
and the $\frac \delta {8 \sqrt{2n-1}}$-Lipschitz  continuity of
$\psi_{\mathcal{T} }$ as well as the  $\frac 3 8 \delta$-Lipschitz
 continuity of $\widetilde{\psi}_{\mathcal{T} _{|_\mathcal{S}}}$ for all $S\in
 \mathcal{S}(\mathcal{T})$.

\medskip

Finally, we verify that $\widetilde{\psi}_{\mathcal{T}}$ has the
desired approximation property. For every $Q\in \mathcal{T}$,
$s\in 2Q$, there exists a point
\begin{displaymath}
s_1 \in Q \cap \Big( E\cup \bigcup_{S\in
\mathcal{S}(\mathcal{T})}\partial S\Big)
\end{displaymath}
with $|s-s_1|\leq |Q|$. This point satisfies by construction
\begin{displaymath}
\psi(s_1)=\widetilde{\psi}_{\mathcal{T}}(s_1)+L_{\mathcal{T}}(s_1).
\end{displaymath}
Then, for $s\in 2Q$ as above, we find that
\begin{align*}
& |\psi(s)-[\widetilde{\psi}_{\mathcal{T}}+L_{\mathcal{T}}](s)|\leq
|\psi(s)-[{\psi}_{\mathcal{T}}+L_{\mathcal{T}}](s)|+|{\psi}_{\mathcal{T}}(s) - \widetilde{\psi}_{\mathcal{T}}(s)|\\
& \leq \frac{\delta |Q|}{8\sqrt{2n-1}} +
|{\psi}_{\mathcal{T}}(s) - \psi_{\mathcal{T}}(s_1)|+ |{\psi}_{\mathcal{T}}(s_1) - \widetilde{\psi}_{\mathcal{T}}(s_1)| + |\widetilde{\psi}_{\mathcal{T}}(s_1) - \widetilde{\psi}_{\mathcal{T}}(s)|\\
& \leq \frac{\delta |Q|}{4\sqrt{2n-1}}+ |{\psi}_{\mathcal{T}}(s_1)
- \widetilde{\psi}_{\mathcal{T}}(s_1)| + \frac{3\delta}{8}|Q|,
\end{align*}
where in the last inequality we used the facts that
$\psi_{\mathcal{T} }$  and $\widetilde{\psi}_{\mathcal{T}
_{|_\mathcal{S}}}$ are $\frac \delta {8 \sqrt{2n-1}}$-Lipschitz
and $\frac 3 8 \delta$-Lipschitz, respectively. Hence it remains
to estimate the term $ |{\psi}_{\mathcal{T}}(s_1) -
\widetilde{\psi}_{\mathcal{T}}(s_1)|.$ There are two cases: $s_1
\in E$ or $s_1 \in \bigcup_{S\in \mathcal{S}(\mathcal{T})}\partial
S.$ In the first case,  $ |{\psi}_{\mathcal{T}}(s_1) -
\widetilde{\psi}_{\mathcal{T}}(s_1)|=0$ because of
\eqref{displaymath1.0}. On the other hand, if $s_1 \in Q \cap
\partial S$ for some $S\in \mathcal{S}(\mathcal{T})$, then
$\widetilde{\psi}_{\mathcal{T}} (s_1)= \psi (s_1) -
L_{\mathcal{T}} (s_1)$ by construction, and so
 \begin{align*}
&  |{\psi}_{\mathcal{T}}(s_1) -
\widetilde{\psi}_{\mathcal{T}}(s_1)| =  |{\psi}_{\mathcal{T}}(s_1)
- (\psi (s_1) - L_{\mathcal{T}} (s_1))| \leq
\frac{\delta}{8\sqrt{2n-1}} |Q|.
\end{align*}
Inserting the bound for $ |{\psi}_{\mathcal{T}}(s_1) -
\widetilde{\psi}_{\mathcal{T}}(s_1)|$ in the previous estimate, we
deduce that
\begin{displaymath}
|\psi(s)-[\widetilde{\psi}_{\mathcal{T}}+L_{\mathcal{T}}](s)|\leq
\frac{\delta|Q|}{4\sqrt{2n-1}}+
\frac{\delta|Q|}{8\sqrt{2n-1}}+\frac{3\delta|Q|}{8}\leq \delta
|Q|,\quad Q\in \mathcal{T},s\in 2Q,
\end{displaymath}
which shows that the approximation condition
\eqref{form15Lipschitz} holds for
$\widetilde{\psi}_{\mathcal{T}}$, and thus concludes the proof of
Corollary \ref{c:DSAdaptBdry}.
\end{proof}

\subsection{Corona decomposition for $1$-dimensional intrinsic
Lipschitz graphs}\label{s:CoronaiLG}

With Corollary \ref{c:DSAdaptBdry} at hand, we are now ready to
establish the corona decomposition for $1$-di{\-}men{\-}sio{\-}nal
intrinsic Lipschitz graphs.  Given such a graph  in
$\mathbb{H}^n$, the idea is to apply Corollary \ref{c:DSAdaptBdry}
to the $1$-dimensional Euclidean Lipschitz graph in
$\mathbb{R}^{2n}$ which is obtained by projecting the intrinsic
graph  to the horizontal coordinate plane $\{t=0\}$. We obtain a
corona decomposition of this graph with approximating Lipschitz
functions
$$\psi_{\mathcal{T}}+L_{\mathcal{T}}:\mathbb{R}\to\mathbb{R}^{2n-1}.$$
While it is straightforward to lift these approximating functions
to intrinsic Lipschitz functions, these lifts may not approximate
the initially given iLG well enough in the $t$-variable.
Analogously to the approach in \cite{fssler2019singular}, the
following theorem is therefore based on a modification of
$\psi_{\mathcal{T}}$ which will ensure that the lift of
$\psi_{\mathcal{T}}+L_{\mathcal{T}}$ has the desired approximation
properties.

\begin{thm}[Corona decomposition for
intrinsic $1$-Lipschitz maps] \label{Theorem 3.15n>1} Let $n>1$,
and assume that $\mathbb{V}$ is a $1$-dimensional horizontal
subgroup in $\mathbb{H}^n$ with complementary vertical subgroup
$\mathbb{W}$. For every $\eta \in (0,1)$, there exists a constant
$C\geq 1$ such that the following holds. Let $$\phi =(\phi_2 ,
\dots, \phi _{2n+1}): \V \to \W$$ be intrinsic 1-Lipschitz. Then,
there exists a coronization $\mathcal{D} = \mathcal{G} \dot{\cup}
\mathcal{B}$ satisfying the conditions in Definition
\ref{d:coronization} with constant $C$ such that,
 for every $\mathcal{T} \in \mathcal{F}$, there is an
  intrinsic Lipschitz map  $\phi_\mathcal{T} =(L_{\mathcal{T}} + \psi_{\mathcal{T}}, \phi_{\mathcal{T}, 2n+1}):\V \to \W$
  where  $L_\mathcal{T}: \R \to \R^{2n-1}$ is a linear 2-Lipschitz map and $\psi _\mathcal{T}: \R \to \R^{2n-1}$
   is   a $\eta$-Lipschitz map such that $\phi_\mathcal{T}$ approximates $\phi$ well at the resolution of the intervals in $\mathcal{T}$:
\begin{equation}\label{FO3.19}
d(\phi (s), \phi_\mathcal{T} (s))\leq \eta |Q|, \quad s \in 2Q, \, Q \in \mathcal{T}.
\end{equation}
\end{thm}
In \eqref{FO3.19}, $d$ refers to left invariant metric on $\He^n$
as defined in \eqref{eq:metric}:
\begin{equation*}
\begin{aligned}
d(\phi (s), \phi_\mathcal{T} (s)) &  = \max \Big\{ | (\phi_2,\dots, \phi_{2n}) (s) - [ L_\mathcal{T} + \psi_\mathcal{T}] (s)|, \\
& \quad  |\phi _{2n+1} (s)- \phi _{\mathcal{T},2n+1} (s) + \frac
12 \sum_{i=2}^n  - \phi_{\mathcal{T},i} (s) \phi_{n+i}(s) +
\phi_{i} (s) \phi_{\mathcal{T},n+i}(s)  |^{1/2}\Big\}.
\end{aligned}
\end{equation*}
If $\Phi$ and
$\Phi_{\mathcal{T}}$ denote the intrinsic graph maps of $\phi$ and
$\phi_{\mathcal{T}}$, respectively, then $d(\phi (s),
\phi_\mathcal{T} (s))= d(\Phi(s),\Phi_{\mathcal{T}}(s))$, so that
\eqref{FO3.19} really gives an estimate on how well the intrinsic
graph of $\phi_{\mathcal{T}}$ approximates the intrinsic graph of
$\phi$.

Theorem  \ref{Theorem 3.15n>1} looks a little different from the
corona decomposition stated in Theorem \ref{t:coronaIntro} in the
introduction. The reason is simply that the graph of $\phi$ and
all the approximating intrinsic graphs in Theorem  \ref{Theorem
3.15n>1} are written as intrinsic graphs of functions
 over the same horizontal subgroup, say they
are all graphs over the $x_1$-axis. Consequently,
$\phi_{\mathcal{T}}$ itself need not have small intrinsic
Lipschitz constant, but its first components have small Lipschitz
constants, up to subtracting the linear term $L_{\mathcal{T}}$.
The following lemma can be used to deduce Theorem
\ref{t:coronaIntro} from Theorem \ref{Theorem 3.15n>1} applied
with constant ``$\eta^2/c_n^2$''.

\begin{lemma}
Let $n\in \mathbb{N}$, and consider the horizontal subgroup
$\mathbb{V}=\{(x_1,0,\ldots,0):\; x_1 \in \mathbb{R}\}$ with
complementary vertical subgroup $\mathbb{W}$. There exists a
constant $1\leq c_n<\infty$, such that the following holds for
every $\eta \in (0,1)$. If
$\phi=(\phi_2,\ldots,\phi_{2n+1}):\mathbb{V}\to \mathbb{W}$ is an
intrinsic Lipschitz function with the property that
\begin{displaymath}
\mathbb{R} \to \mathbb{R}^{2n-1},\quad x\mapsto
(\phi_2,\ldots,\phi_{2n})(x,0,\ldots,0)= \psi(x) + L(x)
\end{displaymath}
is the sum of an $\eta$-Lipschitz map $\psi:\mathbb{R} \to
\mathbb{R}^{2n-1}$ and a linear map $L:\mathbb{R}\to
\mathbb{R}^{2n-1}$, then $\{v\cdot \phi(v):\,v\in\mathbb{V}\}$ is
an intrinsic $c_n \sqrt{\eta}$-Lipschitz graph over the horizontal
subgroup
\begin{displaymath}
\mathbb{V}_L:= \{ (x,L(x),0)\in
\mathbb{R}\times\mathbb{R}^{2n-1}\times \mathbb{R}:\, x\in
\mathbb{R}\}.
\end{displaymath}
\end{lemma}

\begin{proof} We write $\mathbb{V}_L$ as the span of a unit vector
$v_1:=(b_{1,1},\ldots,b_{1,2n})$ in the horizontal plane
$\{t=0\}$. Now we can use similar arguments as in the proof of
Theorem \ref{t:DSAdaptation}. For arbitrary points
$(x,\psi(x)+L(x))$ and $(x',\psi(x')+L(x'))$, we compute
\begin{align*}
|\langle (x-x',&[\psi+L](x)-[\psi+L](x')) ,v_1\rangle |\\&\geq
|(x-x',L(x-x'))|-
|\psi(x)-\psi(x')|\,|(b_{2,1},\ldots,b_{2n,1})|\\
&\geq |x-x'|\left(\sqrt{1+|(b_{2,1},\ldots,b_{2n,1})|^2}-
|(b_{2,1},\ldots,b_{2n,1})|\right)\\
&\geq (\sqrt{2}-1)|x-x'|
\end{align*}
Here, the first inequality can be deduced by triangle inequality
and the fact that the vector $(x-x',L(x-x'))$ is parallel to
$v_1$. The remaining inequalities use that
$v_1=(b_{1,1},\ldots,b_{2n,1})$ is a unit vector, and $\psi$ is
$\eta$-Lipschitz for some $\eta\in (0,1)$. Denoting
\begin{displaymath}
z(x):= \langle (x,[\psi+L](x)) ,v_1\rangle,
\end{displaymath}
the previous computations show
\begin{equation}\label{eq:z_biLip}
|z(x)-z(x')|\geq (\sqrt{2}-1)|x-x'|,
\end{equation}
and $z:\mathbb{R}\to \mathbb{R}$ is a homeomorphism. Now we
complete $v_1$ to an orthonormal basis $\{v_1,\ldots,v_{2n}\}$ of
$\mathbb{R}^{2n}$, and we define a map
$
\varphi: \mathrm{span}(v_1) \to \mathrm{span}\{v_2,\ldots,v_{2n}\}
$
by setting
\begin{displaymath}
\varphi(z(x) v_1 ) = \sum_{j=2}^{2n} \left\langle
\begin{pmatrix}x \\ [\psi+L](x) \end{pmatrix} ,v_j\right \rangle v_j
\end{displaymath}
so that the graph of $\varphi$ (over $\mathbb{V}_L$) as a set in
$\mathbb{R}^{2n}$ coincides with graph of $\psi+L$ (over
$\mathbb{V}$). Then there exists a constant $\lambda_n$, depending
only on $n$, such that
\begin{align*}
|\varphi(z(x)v_1)-\varphi(z(x')v_1)|\leq \lambda_n
|\psi(x)-\psi(x')|\leq \lambda_n \eta |x-x'|\leq
\frac{\lambda_n\,\eta }{\sqrt{2}-1} |z(x)-z(x')|.
\end{align*}
This shows that the projection of the intrinsic graph $\Gamma$ of
$\phi$ to the horizontal plane $\{t=0\}$ is the graph of the
Euclidean Lipschitz function $\varphi$ over $\mathrm{span}(v_1)$.
Then it is easy to see that there exists a unique real-valued
function $\varphi_{2n+1}$ so that $\Gamma$ is the intrinsic graph
of $(\varphi,\varphi_{2n+1})$ over $\mathbb{V}_L$, and the graph
map of $\phi$ at $x$ equals the graph map of
$(\varphi,\varphi_{2n+1})$ at $z(x)v_1$. It remains to show that
$(\varphi,\varphi_{2n+1})$ is intrinsic $c_n\sqrt{\eta}$-Lipschitz
for a suitable constant $c_n$. It is easy to deduce from the
intrinsic Lipschitz property of $\phi$, Remark \ref{r:iLGvsLip}
applied to $\phi$, and \eqref{eq:z_biLip} that the graph map of
$(\varphi,\varphi_{2n+1})$ is a Lipschitz function with respect to
the Heisenberg metric. Applying again Remark \ref{r:iLGvsLip}, but
now to $(\varphi,\varphi_{2n+1})$, we conclude that this function
is intrinsic Lipschitz. Finally, it follows from the Euclidean
$(\lambda_n\,\eta)/(\sqrt{2}-1)$-Lipschitz continuity and the
arguments in Section \ref{s:defns} that the intrinsic Lipschitz
constant of $(\varphi,\varphi_{2n+1})$ can be taken to be $c_n
\sqrt{\eta}$.
\end{proof}

\medskip

Before proving Theorem \ref{Theorem 3.15n>1}, we give another
version for intrinsic $N$-Lipschitz maps
$\phi=(\phi_2,\ldots,\phi_{2n+1})$ with $N \geq 1$, similarly as
in \cite{fssler2019singular}. We can consider
\begin{displaymath}
\hat \phi = (\hat \phi_{2}, \dots, \hat \phi_{2n+1}
)=\left(\tfrac{1}{N}\phi_2,\ldots,\tfrac{1}{N}
\phi_n,\tfrac{1}{N^2}\phi_{n+1},\tfrac{1}{N}\phi_{n+2},\ldots,\tfrac{1}{N}\phi_{2n},\tfrac{1}{N^2}\phi_{2n+1}\right),
\end{displaymath} which is an intrinsic $1$-Lipschitz map. Hence, we apply Theorem
\ref{Theorem 3.15n>1} to $\hat \phi$ and constant $\eta$. This
yields a coronization with Carleson packing constants independent
of $N$, and for every associated tree $\mathcal{T}$ an
approximating map
$\hat\phi_{\mathcal{T}}=(\hat\psi_{\mathcal{T}}+\hat{L}_{\mathcal{T}},\hat\phi_{\mathcal{T},2n+1})$
as stated in Theorem \ref{Theorem 3.15n>1}. Then
\begin{displaymath}
 \phi_{\mathcal{T}}:
=\left(N\hat\phi_{\mathcal{T},2},\ldots,N\hat
\phi_{\mathcal{T},n},N^2\hat\phi_{\mathcal{T},n+1},N\hat\phi_{\mathcal{T},n+2},\ldots,N\hat\phi_{\mathcal{T},2n},N^2\hat\phi_{\mathcal{T},2n+1}\right)
\end{displaymath}
is intrinsic Lipschitz and its projection to the horizontal plane
$\{t=0\}$ is the sum of a $\eta N^2$-Lipschitz map
$\psi_{\mathcal{T}}$ and a linear $2 N^2$-Lipschitz map
$L_{\mathcal{T}}$ with the properties stated in the following
corollary. The appearance of the constant $N^2$ (instead of $N$)
is related to the fact that intrinsic $N$-Lipschitz maps
correspond essentially to $(N,\ldots,N,N^2)$-tame maps by
Propositions \ref{p:FromIntrLipToTameGENERAL} and
\ref{p:FromTameToIntrLip}.

\begin{cor}[Corona decomposition for
intrinsic $N$-Lipschitz maps]\label{Theorem 3.15+} For every $n\in
\mathbb{N}$, $n>1$, and $\eta \in (0,1)$, there exists a constant
$C\geq 1$ such that the following holds. Let  $N\geq 1$ be
arbitrary and let $\phi =(\phi_2 , \dots, \phi _{2n+1}): \V \to
\W$ be intrinsic $N$-Lipschitz. Then, there exists a coronization
$\mathcal{D} = \mathcal{G} \dot{\cup} \mathcal{B}$ satisfying the
conditions in Definition \ref{d:coronization} with constant $C$
such that,
 for every $\mathcal{T} \in \mathcal{F}$, there is an
  intrinsic Lipschitz map  $\phi_\mathcal{T} =(L_{\mathcal{T}} + \psi_{\mathcal{T}}, \phi_{\mathcal{T}, 2n+1}):\V \to \W$
  where  $L_\mathcal{T}: \R \to \R^{2n-1}$ is a linear $2N^2$-Lipschitz map and $\psi _\mathcal{T}: \R \to \R^{2n-1}$
   is   a $\eta N^2$-Lipschitz map such that $\phi_\mathcal{T}$ approximates $\phi$ well at the resolution of the intervals in $\mathcal{T}$:
\begin{equation*}
d(\phi (s), \phi_\mathcal{T} (s))\leq (\eta N^2) |Q|, \quad s \in 2Q, \, Q \in \mathcal{T}.
\end{equation*}
\end{cor}

\begin{proof}[Proof of Theorem \ref{Theorem 3.15n>1}]
 We apply the Lipschitz
corona decomposition stated in Theorem \ref{t:DSAdaptation} and
Corollary \ref{c:DSAdaptBdry} with
 parameter $\delta :=  \eta^2/(100 n)$
 to the $1$-Lipschitz map  $\psi:=( \phi_{2} , \dots , \phi_{2n}):\R \to \R ^{2n-1} $.
 Hence, there are a coronization with Carleson packing constant
 depending on $n$ and $\eta$, and an associated forest $\mathcal{F}$ of trees. We fix $\mathcal{T}\in \mathcal{F}$, and consider the top
interval $Q(\mathcal{T})=[x,y]$ with $x<y$. Then there exists a
 $\delta$-Lipschitz map $\psi_{\mathcal{T}}= ( \psi_{\mathcal{T},2} , \dots , \psi_{\mathcal{T},2n}):\R \to \R ^{2n-1} $   and a linear 2-Lipschitz map $L_{\mathcal{T}}=( L_{\mathcal{T},2} , \dots , L_{\mathcal{T},2n}):\R \to \R ^{2n-1}$  such that
\begin{equation}\label{FO3.24A}
|(\phi_2,\dots, \phi_{2n}) (s) - [ L_\mathcal{T} + \psi_\mathcal{T}] (s)| \leq \delta |Q|, \quad s \in 2Q, \, Q \in \mathcal{T}.
\end{equation}
In addition, we may assume by Corollary \ref{c:DSAdaptBdry} that
\begin{equation}\label{FO3.24}
\phi_i (s) = [ L_{\mathcal{T},i} + \psi_{\mathcal{T},i}] (s),
\quad \text{for all }s \in E\cup \bigcup_{S\in
\mathcal{S}(\mathcal{T})}\partial S, \quad i=2,\dots , 2n.
\end{equation}
Here, as before, $\mathcal{S} (\mathcal{T})$ is the collection of
minimal intervals in $\mathcal{T}$ (possibly an empty collection)
and
\begin{equation*}
E =  Q(\mathcal{T})\setminus \bigcup _{ S \in \mathcal{S}
(\mathcal{T})} S.
\end{equation*}
Now we would like to produce an intrinsic Lipschitz function
\begin{equation}\label{eq:phiTDef}
\phi_\mathcal{T} =(L_{\mathcal{T},2} + \psi_{\mathcal{T},2} ,
\ldots ,L_{\mathcal{T},2n} + \psi_{\mathcal{T},2n},
\phi_{\mathcal{T}, 2n+1}):\V \to \W \end{equation} satisfying the
claims stated in Theorem \ref{Theorem 3.15n>1}. The challenge is
to find the last component of $\phi_{\mathcal{T}}$ so that the
intrinsic Lipschitz and approximation property hold, and this will
require some changes in the terms $L_{\mathcal{T},n+1} +
\psi_{\mathcal{T},n+1}$ (but not in the other components).

For $S=[a,b] \in \mathcal{S} (\mathcal{T})$ fixed, we will modify
the restriction of $\psi_{\mathcal{T},n+1}$ to $\frac 1 2 S
=[s_1,s_2]$ with $s_1\leq s_2$, which is the interval with the
same center but half the length as $S$. The property of $\frac 1 2
S$ needed in the future is that if $Q \in \mathcal{T}$ with $|Q| <
|S|$, then
\begin{equation}\label{3.26}
2Q \cap \tfrac{ 1}{ 2} S = \emptyset .
\end{equation}

Analogously as in the proof of \cite[Theorem 3.15]{fssler2019singular}, we add to
$\psi_{\mathcal{T},n+1}$ a suitable ``correction term''
$\xi_S : S \to \R$ in order that
\begin{enumerate}
\item $\xi_S (t)=0, \quad \forall t \ne [s_1,s_2]$;
\item it holds
\begin{align}\label{FO3.27}
\int_a^b - \xi _S (r)\, dr&= \int _a^b - \phi _{n+1} (r) +
\psi_{\mathcal{T},n+1}(r) + L_{\mathcal{T},n+1}(r)\\&+\frac 12
\sum_{i=2}^n  \phi_i (r) \dot \phi_{n+i} (r)-\dot \phi_i (r)
\phi_{n+i}(r )  -  \phi_{\mathcal{T},i} (r) \dot
\phi_{\mathcal{T},n+i} (r)-\dot \phi_{\mathcal{T},i} (r)
\phi_{\mathcal{T},n+i}(r )  \, d r.\notag
\end{align}
\end{enumerate}
The idea behind \eqref{FO3.27} is the following. As suggested by
\eqref{eq:phiTDef}, we will define
\begin{displaymath}
\phi_{\mathcal{T},i}:=
\psi_{\mathcal{T},i}+L_{\mathcal{T},i},\quad \text{for
}i=2,\ldots,n,n+2,\ldots,2n,
\end{displaymath}
but for $i=n+1$ and $S\in \mathcal{S}(\mathcal{T})$, we set
\begin{displaymath}
\phi_{\mathcal{T},n+1}|_{S}:=
\psi_{\mathcal{T},n+1}|_S+\xi_S+L_{\mathcal{T},n+1}|_S,
\end{displaymath}
while $\phi_{\mathcal{T},n+1}|_{E}:=\psi_{\mathcal{T},n+1}|_E +
L_{\mathcal{T},n+1}|_E$. The function $\xi_S$ allows us to match
$\phi_{\mathcal{T},2n+1}$ with $\phi_{2n+1}$ in endpoints of
minimal intervals. Up to a sign change, the desired intrinsic
Lipschitz property of $\phi_{\mathcal{T}}$ is equivalent to the
tameness condition. Tame maps on intervals can be characterized as
in Proposition \ref{propositionk=1}, so we will obtain
$\phi_{\mathcal{T},2n+1}$ by integrating an expression involving
the components
$\phi_{\mathcal{T},2},\ldots,\psi_{\mathcal{T},2n}$. Then
\eqref{FO3.27} ensures that the thus defined
$\phi_{\mathcal{T},2n+1}$ agrees with $\phi_{2n+1}$ in endpoints
of the minimal intervals $S\in\mathcal{S}(\mathcal{T})$.

To obtain \eqref{FO3.27}, we define $\xi_S : S \to \R$ as
\begin{equation*}
\begin{aligned}
 \xi _S(t) &:= \left\{
\begin{array}{l}
4c(t-s_1), \qquad \,\, \quad  \mbox{ for } t \in [s_1, \frac{s_1+s_2}{2}  ],\\
4c(s_2-t), \quad \qquad \,\,\mbox{ for } t \in (\frac{s_1+s_2}{2} , s_2],\\
0, \quad \qquad \qquad \qquad \mbox{  otherwise }
\end{array}
\right.
\end{aligned}
\end{equation*}
where $c\in \R$ is such that \eqref{FO3.27} holds. Since
\begin{equation*}
\begin{aligned}
\int_S \xi _S (r)\, dr & = \frac{c |S|^2}{4},
\end{aligned}
\end{equation*}
and $S=[a,b]$, the requirement \eqref{FO3.27} means that
\begin{equation}\label{ceq32}
\begin{aligned}
-c = \frac 4 {(b-a)^2}& \int _a^b - \phi _{n+1} (r) +
\psi_{\mathcal{T},n+1}(r) + L_{\mathcal{T},n+1}(r)\\& +\frac 12
\sum_{i=2}^n  \phi_i (r) \dot \phi_{n+i} (r)-\dot \phi_i (r)
\phi_{n+i}(r )    -  \phi_{\mathcal{T},i} (r) \dot
\phi_{\mathcal{T},n+i} (r)+\dot \phi_{\mathcal{T},i} (r)
\phi_{\mathcal{T},n+i}(r )  \, d r.
\end{aligned}
\end{equation}
The $\eta$-Lipschitz continuity of
\begin{equation}\label{eq:PsiModif}
(\psi_{\mathcal{T},2},\ldots,\psi_{\mathcal{T},n},\psi_{\mathcal{T},n+1}+\xi_S,
\psi_{\mathcal{T},n+2},\ldots,\psi_{\mathcal{T},2n})
\end{equation}
on $S$ will follow from a bound on $|c|$. We claim that $|c|\leq
24 n \delta$; indeed, by \eqref{ceq32},
\begin{equation}\label{boundc21}
\begin{aligned}
-c & = \frac 4 {(b-a)^2} \int _a^b - \phi _{n+1} (r) + \psi_{\mathcal{T},n+1}(r) + L_{\mathcal{T},n+1}(r) \, dr \\
&  + \frac 2 {(b-a)^2}  \sum_{i=2}^n \int _a^b  ( \phi_i (r) - \phi_{\mathcal{T},i}(r)) \dot \phi_{n+i} (r)-  ( \phi_{n+i} (r) - \phi_{\mathcal{T},n+i}(r)) \dot \phi_{i} (r)\, dr \\
&  +\frac 2 {(b-a)^2}  \sum_{i=2}^n \int _a^b  ( \phi_i (r) - \phi_{\mathcal{T},i}(r)) \dot \phi_{\mathcal{T}, n+i} (r)-  ( \phi_{n+i} (r) - \phi_{\mathcal{T},n+i}(r)) \dot \phi_{\mathcal{T}, i} (r)\, dr \\
&  + \frac 2 {(b-a)^2}  \sum_{i=2}^n \int _a^b   \phi_{\mathcal{T},i}(r) \dot \phi_{ n+i} (r)
- \phi_{\mathcal{T},n+i}(r) \dot \phi_{ i} (r) - \phi_{i}(r) \dot \phi_{\mathcal{T}, n+i} (r)
+ \phi_{n+i}(r) \dot \phi_{\mathcal{T}, i} (r) \, dr \\
&  =: I_1+I_2+I_3+I_4.
\end{aligned}
\end{equation}
Using \eqref{FO3.24A}, we obtain that $|I_1|\leq 4 \delta $;
moreover, using again \eqref{FO3.24A} and recalling that
$\psi_{\mathcal{T}} $ is a $\delta$-Lipschitz map with $\delta <1$
and $L_{\mathcal{T}}$ is a linear 2-Lipschitz map, we have that
$|I_i|\leq 12 (n-1) \delta$ for $i=2,3$. Finally, integrating by
parts, and using $\phi_i(s)= \phi_{\mathcal{T},i}(s)$ for $s\in
\{a,b\}$ and $i=2,\ldots,2n$, we get that $I_4=0$.
 Hence $|c|\leq 24 n \delta$, as desired.
Consequently,  we get that $\xi_S$ is $96 n \delta$-Lipschitz
with
 $\|\xi _S\|_{L^\infty} \leq 24 n \delta |S|$.
We make analogous modifications inside all intervals  $S \in
\mathcal{S} (\mathcal{T})$. Recalling that we have chosen $\delta$
so that $100 n \delta = \eta^2\leq \eta$, we
 obtain an $\eta$-Lipschitz map on $\mathbb{R}$, piecewise
defined on $S$ as in \eqref{eq:PsiModif}.

\medskip

We next show that the modified map still satisfies
\eqref{FO3.24A}, albeit with a larger constant than $\delta$.
Indeed, for $Q\in \mathcal{T}$ it suffices to check the condition
for $s\in 2Q$ that belong to $\frac{1}{2}S$ for some $S\in
\mathcal{S}(\mathcal{T})$, as this is the only place where we have
done a modification. So assume $s\in 2Q \cap \frac{1}{2}S$. Then
$|S|\leq |Q|$ by \eqref{3.26}, and \eqref{FO3.24A} yields
\begin{equation}\label{FO3.24B}
|(\phi_2,\ldots,\phi_{2n})(s)-(\phi_2,\ldots,\phi_{2n})(s)| \leq
25 n \delta |Q| \leq \eta |Q|.
\end{equation}

Now we consider the last component of the approximation map
$\phi_{\mathcal{T}}$ of $\phi$. For $Q(\mathcal{T})=[x,y]$, we
define
\begin{equation*}
\phi _{\mathcal{T},2n+1} (s) := \phi _{2n+1} (x) + \int_x^s -
\phi_{\mathcal{T},n+1} (r) +\frac 12 \sum_{i=2}^n
\phi_{\mathcal{T},i} (r) \dot \phi_{\mathcal{T},n+i} (r)- \dot
\phi_{\mathcal{T},i} (r) \phi_{\mathcal{T},n+i} (r)\, dr\\,
\end{equation*}
for all $s \in [x,y].$ By Proposition \ref{propositionk=1}, $\phi
_{\mathcal{T}}: \V \to \W$ is an intrinsic Lipschitz map. The next
step is to show
\begin{equation}\label{claim1parteequi}
\phi_{2n+1} (s)
 =
\phi _{\mathcal{T},2n+1} (s),
\end{equation}
for $s \in E \cup \bigcup _{S \in \mathcal{S} (\mathcal{T})}
\partial S$. By construction, this is equivalent to verifying that for
 $s \in E \cup \bigcup _{S \in \mathcal{S} (\mathcal{T})} \partial S$ and $Q(\mathcal{T})=[x,y]$ we have
\begin{equation*}
 \int_{x}^s  \dot \phi_{2n+1}(r) \, dr = \int_{x}^s  \dot \phi_{\mathcal{T},2n+1} (r)\,
 dr.
\end{equation*}
 We recall that $E$ is a measurable set because $E=
Q(\mathcal{T})\setminus \bigcup _{ S \in \mathcal{S}
(\mathcal{T})} S$ and $\mathcal{S} (\mathcal{T})$ is a countable
family of intervals. Moreover,
\begin{equation*}
\phi _{\mathcal{T},2n+1} (s)
=  \phi_{2n+1} (x)
 + \int_{ E  \cap [x,s]} \dot \phi _{\mathcal{T},2n+1} (r)\, dr  + \sum_{ S \in \mathcal{S} (\mathcal{T})}  \int_{   S \cap [x,s]} \dot \phi _{\mathcal{T},2n+1} (r)\, dr.\\
\end{equation*}
If $E$ is a Lebesgue null set, then the integral over $E$ is not
relevant. On the other hand, the derivatives of
$\phi_i,\phi_{\mathcal{T},i}$, $i=2,\ldots,2n+1$ exist almost
everywhere, and if $E$ has positive measure, then almost every
point in $E$ is a Lebesgue density point of $E$. Since
$\phi_i(s)=\phi_{\mathcal{T},i}(s)$ for all $s\in E$ and
$i=2,\ldots,2n$, it follows that $\dot \phi_{i}(s)=\dot
\phi_{\mathcal{T},i} (s)$ for almost every $s\in E$ and so
$
 \dot \phi _{\mathcal{T},2n+1} (s)=  \dot \phi _{2n+1}(s)$ for almost every $s\in E$.
Moreover,
\begin{equation}\label{claim1parteS}
\begin{aligned}
 \sum_{ S \in \mathcal{S} (\mathcal{T})}  \int_{   S \cap [x,s]} \dot \phi _{\mathcal{T},2n+1} (r)\, dr = \sum_{ S \in \mathcal{S} (\mathcal{T})}  \int_{   S \cap [x,s]} \dot \phi _{2n+1} (r)\, dr .
\end{aligned}
\end{equation}
Indeed, by the choice of $s\in E \cup \bigcup \partial S$ we have
two cases to consider: $ S \cap [x,s]= S$ or $ S \cap [x,s]
=\emptyset.$ The latter intervals $S$ can be ignored, and for the
former, the integrals on the left and on the right-hand side of
\eqref{claim1parteS} agree, by \eqref{FO3.27}. This proves
\eqref{claim1parteequi}.

\medskip

Finally, it remains to check the approximation condition
\eqref{FO3.19}. Having already established \eqref{FO3.24B}, the
only nontrivial inequality is
\begin{equation}\label{FO3.29finale}
A:=\left| \phi _{2n+1} (s)- \phi _{\mathcal{T},2n+1} (s) + \frac
12 \sum_{i=2}^n  - \phi_{\mathcal{T},i} (s) \phi_{n+i}(s) +
\phi_{i} (s) \phi_{\mathcal{T},n+i}(s) \right| \leq \eta ^2 |Q|^2,
\end{equation}
for every $Q\in \mathcal{T}$, $s\in 2Q$. We have two different
cases: $s\in E$ and $s\notin E.$ Firstly, for $s\in E$, the
left-hand side of \eqref{FO3.29finale} vanishes by \eqref{FO3.24}
and \eqref{claim1parteequi}, and so the inequality holds trivially
true. Secondly, for $s\notin E$, as in the proof of
\cite[Proposition 3.6]{fssler2019singular}, we know that there is
$\tilde s \in Q \cap (E \cup \bigcup _{ S \in \mathcal{S}
(\mathcal{T})} \partial S)$ such that $|s-\tilde s | \leq |Q|$.
Since $\phi _{\mathcal{T},n+1} (\tilde s) = \phi_{n+1}(\tilde s)$,
we can estimate as follows:
\begin{align*}
A
 &\leq  \Biggl|\phi _{2n+1} (s)-\phi _{2n+1} (\tilde s)
 - \phi _{\mathcal{T},2n+1} (s) +\phi _{\mathcal{T},2n+1} (\tilde s)  + \tfrac{ 1}{2} \sum_{i=2}^n  - \phi_{\mathcal{T},i} (s) \phi_{n+i}(s) + \phi_{i} (s) \phi_{\mathcal{T},n+i}(s) \Biggl| \\
&=  \Biggl| \int_{\tilde s}^s  - \phi _{n+1} (r) +
\phi _{\mathcal{T},n+1} (r)
+\tfrac{ 1}{2} \sum_{i=2}^n  \phi_i (r) \dot \phi_{n+i} (r)-\dot \phi_i (r) \phi_{n+i}(r ) \\
& \quad\quad\quad\quad\quad\quad\quad\quad\quad\quad\quad\quad\quad- \tfrac{ 1}{2} \sum_{i=2}^n  \phi_{\mathcal{T},i} (r) \dot \phi_{\mathcal{T},n+i} (r) -  \phi_{\mathcal{T},n+i} (r) \dot \phi_{\mathcal{T},i} (r)\, dr\\
&\quad +\tfrac{ 1}{2} \sum_{i=2}^n
 \phi_{\mathcal{T},n+i} ( s)  \phi_{i} (s)
- \phi_{\mathcal{T},i} ( s)  \phi_{ n+i} ( s) \Biggl|.
\end{align*}
For convenience, let us denote the integral in the above
expression by ``$I$'', so that
\begin{equation}\label{eq:EstWithI}
A\leq \left|I + \tfrac{ 1}{2} \sum_{i=2}^n
 \phi_{\mathcal{T},n+i} ( s)  \phi_{i} (s)
- \phi_{\mathcal{T},i} ( s)  \phi_{ n+i} ( s) \right|.
\end{equation}
Using similar computations as in the bound for $c$, see
\eqref{boundc21}, we find
\begin{align*}
I= & \int_{\tilde s}^s - \phi _{n+1} (r) + \psi_{\mathcal{T},n+1}(r)  \, dr \\
&  + \tfrac{1}{2}  \sum_{i=2}^n \int_{\tilde s}^s  ( \phi_i (r) - \phi_{\mathcal{T},i}(r)) \dot \phi_{n+i} (r)-  ( \phi_{n+i} (r) - \phi_{\mathcal{T},n+i}(r)) \dot \phi_{i} (r)\, dr \\
&  +\tfrac{1}{2} \sum_{i=2}^n \int_{\tilde s}^s  ( \phi_i (r) - \phi_{\mathcal{T},i}(r)) \dot \phi_{\mathcal{T}, n+i} (r)-  ( \phi_{n+i} (r) - \phi_{\mathcal{T},n+i}(r)) \dot \phi_{\mathcal{T}, i} (r)\, dr \\
&  + \tfrac{1}{2}  \sum_{i=2}^n \int_{\tilde s}^s
\phi_{\mathcal{T},i}(r) \dot \phi_{ n+i} (r) -
\phi_{\mathcal{T},n+i}(r) \dot \phi_{ i} (r) - \phi_{i}(r) \dot
\phi_{\mathcal{T}, n+i} (r) + \phi_{n+i}(r) \dot
\phi_{\mathcal{T}, i} (r) \, dr\\
=:& J_1 + J_2 + J_3 + J_4.
\end{align*}
Notice that $[\tilde s, s] \subset 2Q$ (or $[s, \tilde s] \subset
2Q$), so $|\phi_{\mathcal{T},n+1}(r)-\phi_{n+1}(r)|\leq 25 n
\delta |Q|$ for all $r\in [\tilde s,s]$ by \eqref{FO3.24B}, and
$|\phi_{\mathcal{T},i}(r)-\phi_{i}(r)|\leq \delta |Q|$ for $i\in
\{2,\ldots,2n\}\setminus \{n+1\}$ by the property coming from
Theorem \ref{t:DSAdaptation}. Using also that $|s-\tilde s|\leq
|Q|$, we obtain the desired estimates for the first three
summands: $ |J_1|+|J_2|+|J_3| \leq 50 n \delta |Q|^2. $ The term
``$J_4$'' might look problematic at first since $\Phi(s)$ does not
necessarily agree with $\Phi_{\mathcal{T}}(s)$. However, if we
combine it first with the second summand in \eqref{eq:EstWithI},
then cancellation occurs by partial integration:
\begin{align*}
J_4 + \tfrac{ 1}{2} \sum_{i=2}^n
 \phi_{\mathcal{T},n+i} ( s)  \phi_{i} (s)
- \phi_{\mathcal{T},i} ( s)  \phi_{ n+i} ( s)
=\tfrac{1}{2}\sum_{i=2}^n
-\phi_{\mathcal{T},i}(\widetilde{s})\phi_{n+i}(\widetilde{s})
+\phi_{\mathcal{T},n+i}(\widetilde{s})\phi_i(\widetilde{s})
.
\end{align*}
Since $\Phi(\widetilde s) = \Phi_{\mathcal{T}}(\widetilde s)$, the
expression on the right vanishes. Hence we obtain that
\begin{equation*}
A\leq 50 n \delta |Q|^2.
\end{equation*}
Finally, we recall that $10 0 n \delta = \eta^2$, so
\eqref{FO3.29finale} holds, as desired. This concludes the proof.
\end{proof}

\bibliographystyle{plain}%

\bibliography{references}

\end{document}